\documentclass[11pt,a4paper]{article}
\usepackage[top=2.5cm,bottom=2.5cm,left=2.8cm,right=2.8cm]{geometry}
\usepackage[dvipsnames,svgnames,x11names]{xcolor}
\usepackage{float}
\usepackage{graphicx}
\usepackage{amsmath,amssymb,mathtools}
\usepackage{amsthm}
\usepackage{fancyhdr}
\usepackage{hyperref}
\usepackage{subfigure}
\usepackage{tikz}
\usetikzlibrary{decorations.fractals,shapes.geometric,calc}
\usetikzlibrary{arrows.meta}

\usepackage[english]{babel}
\usepackage[backend=biber, style=numeric, sorting=nyt]{biblatex}
\usepackage{csquotes}

\hypersetup{
    pdfauthor={FLAVIO MOSCA},
    pdftitle={A PROBLEM FOR ISOLATED SINGULARITIES OF SURFACES},
    colorlinks=true,
    linkcolor=RoyalBlue,
    citecolor=RoyalBlue,
    urlcolor=RoyalBlue,
    filecolor=RoyalBlue,
    linktoc=all
}

\theoremstyle{plain}
\newtheorem{theorem}{Theorem}[section]
\newtheorem{lemma}[theorem]{Lemma}
\newtheorem{corollary}[theorem]{Corollary}
\newtheorem{addendum}[theorem]{Addendum}
\theoremstyle{definition}
\newtheorem{definition}[theorem]{Definition}
\theoremstyle{remark}
\newtheorem{remark}[theorem]{Remark}

\renewcommand{\S}{\ensuremath{\mathsection}}
\begin{document}

\begin{center}
    \textbf{\large A PROBLEM FOR ISOLATED SINGULARITIES OF SURFACES.}
    
    \vspace{0.8em}
    \textbf{FLAVIO MOSCA}
\end{center}

\noindent\textbf{Abstract.} Let $E\subset\mathbf R^n$ be a set of dimension $2$ with an isolated singular point $0$, such that $E-\{0\}$ is an orientable $\mathcal C^2$ submanifold of $\mathbf R^n$ of dimension $2$, and hence a Riemann surface with the conformal structure induced by the ambient metric. We study the extension of its conformal structure across $0$, under hypotheses that are metric rather than analytic. Denoting by $S(0,r)$ the sphere of radius $r$ centred at $0$, we call the singularity $0$ $\mathit{M\text{-}regular}$ if (i) the pair $(E-\{0\},\{0\})$ satisfies a Whitney condition (the radial direction becomes asymptotically tangent to $E-\{0\}$ at $0$) and (ii) the length of the spherical slice $S(0,r)\cap E$ is bounded above by a linear function of $r$ for small $r$ ($\mathit{metric\ decay}$). The main result (1989) is that an $M$-regular singularity is $\mathit{conformally\ point\text{-}like}$, that is, a punctured neighbourhood of $0$ in $E$ is conformal to the punctured unit disc $\Delta^*$, so that the conformal structure extends across $0$. The proof uses the modulus of rings (extremal length), which generalizes the classical planar technique to $2$-dimensional submanifolds of $\mathbb{R}^n$. Two classes of surfaces are shown to be conformally point-like: all surfaces of revolution generated by a single $C^1$ curve, and orientable subanalytic sets of dimension $2$ with isolated singularities, for which $M$-regularity is proven directly from Hironaka's structure theory. In this case the metric decay condition follows from a direct estimate of the volume of the slices $S(0,r)\cap E$ (Lemma 7.5), based on Hironaka's Whitney theorem and Stasica's decomposition into $L$-analytic pieces. A further original result is a $C^\infty$ counterexample showing that a strict Whitney condition alone does $\mathit{not}$ imply the linear length bound: the two hypotheses defining $M$-regularity are independent.

\tableofcontents

\section{Introduction}\label{sec:intro}

This note presents a problem of extending the complex structure of Riemann surfaces, obtained from submanifolds of $\mathbf{R}^n$ of dimension $2$, to an isolated singular point. This 2026 edition revises and annotates the original 1989 preprint \cite{Mo89}, which remains the primary source for the results below. The hypotheses on the singularities are essentially metric, and the class of surfaces satisfying them is sufficiently broad (for instance, it includes subanalytic surfaces).

Let us first mention that it was originally assumed that the hypothesis of analyticity of the surfaces was essential for this extension of the complex structure, whereas this hypothesis plays no role in the solution and is more restrictive than necessary: the attempted solutions were initially based on analyticity, which was abandoned only at the end. We now recall the original formulation of the problem. Let $E$ be an analytic subset of $\mathbf{R}^n$, such that $V = (E - \{0\}) \cap B(0, r)$ is an orientable analytic submanifold of $\mathbf{R}^n$ with $\dim V = 2$. Consider the metric $ds^2$ induced by $\mathbf{R}^n$ on $V$: there then exists a conformal structure on $V$ compatible with this metric, in other words a set of local parametrizations $\mathbf{C} \supset U \to V$ such that $ds^2 = \lambda dw d\overline{w}$ (isothermal parameters) \cite[\S2.5]{AS}.

$V$ thus receives a natural Riemann surface structure, and it is therefore interesting to ask under what hypotheses these surfaces are all conformally equivalent to one another or not; it is clear that if all the $V$'s are conformal to each other then a connected representative is the unit disc minus the origin $\Delta^*$.

Let $V$ be connected: from \cite[Ch.~IV, \S6.1]{FK} we deduce that there are only 2 possibilities for the type of $V$, since all Riemann surfaces with $\Pi_1(V) = \mathbf{Z}$ (as in our case, given the previous observation) are conformally equivalent to $\mathbf{C}^*$, $\Delta^*$, or $A_r = \{z \in \mathbf{C} \mid r < |z| < 1, 1 > r > 0\}$. We can immediately discard $\mathbf{C}^*$ by taking $r' < r$, i.e.\ a $V'$ relatively compact in $V$, and we are left with the choice between $V \simeq \Delta^*$ and $V \simeq A_r$. What we expect is that $V$ is always conformal to the punctured disc or, in other words, that the natural conformal structure of $V$ can be uniquely extended to the singular point. Using well-known theorems on holomorphic or harmonic functions in the plane, one sees that the statement to be proved can be expressed in a more intrinsic way as follows: every holomorphic or harmonic function on $V$ that is bounded on a set $B(0, r) \cap V$ can be extended at 0.

In this paper every surface is a set of dimension $2$. Throughout what follows we shall consider orientable $C^2$ submanifolds of $\mathbf{R}^n$ of dimension 2 that have a natural Riemann surface structure induced by the ambient metric: see for instance \cite{Ch}. We shall not be concerned here with minimal differentiability hypotheses. We are particularly interested in surfaces with isolated singularities, that is, submanifolds $S$ of class $C^0$ such that $\exists\ P \in S$ for which $S - \{P\}$ is a $C^2$ manifold. Under the hypotheses: 1) the pair $(S - \{P\}, \{P\})$ satisfies a Whitney condition, 2) the length of the locus of points in $S$ at Euclidean distance $r$ from $P$ is bounded above by a function of the type $Cr$ with $C$ a constant (a condition described as ``metric decay''), we shall prove that all connected components $V$ of a sufficiently small neighborhood of $P$ in $S$, with $H_1(V) = \mathbf{Z}$, are conformal to $\Delta^*$.

This result is proved in Section~\ref{sec:main}, preceded by an introduction to the Whitney condition in Section~\ref{sec:whitney}. The proof relies on inequality~\ref{lem:3.5} for the moduli of rings, which condition 2) of ``metric decay'' allows us to exploit. The Whitney condition is essentially needed to show that suitable ``slices'' of the surface are homeomorphic to rings via Lemma~\ref{lem:4.5}. The same result, in the different setting of surfaces of revolution generated by a curve, is proved in Section~\ref{sec:rev}.

We explicitly note that in this case the singularity is not necessarily isolated, since we allow curves with infinitely many accumulation points on the $z$-axis (for example $(x, \sin \frac{1}{x}), x > 0$).

Section~\ref{sec:counter} presents an example of a $C^\infty$ surface in $\mathbf{R}^3$ with an isolated singularity that satisfies a strict Whitney condition but not condition 2) of metric decay, showing that the two conditions are independent. As a simple example of a surface of revolution that satisfies 2) but not 1) we can consider the surface: $(x, y, \sqrt{x^2 + y^2} \sin \frac{1}{\sqrt{x^2 + y^2}})$ .

Section~\ref{sec:subanalytic} applies the results obtained to subanalytic surfaces.

This 2026 edition revises the original 1989 preprint \cite{Mo89} and integrates a number of corrections and additions. Corrections concern only the rigor and generality of certain proofs, while all results stated in 1989 are confirmed to hold.

\section{Surfaces of revolution}\label{sec:rev}

The construction below extends to $\mathbf{R}^n$ the standard construction of a surface of revolution in $\mathbf{R}^3$ \cite{dC}, combined with the classical fact that a change of variable of the type used here exhibits a surface of revolution as conformal to a punctured disc or an annulus \cite{Ch}, \cite{AS}.

Let $\Gamma : {]0, 1]} \to \mathbf{R}^{n-1} = \{x_2, ..., x_{n-1}, z\} \subset \mathbf{R}^n$ be a $C^1$ curve with $\Gamma \cap \{z\text{-axis}\} = \emptyset$, and let $C = \overline{\Gamma} - \Gamma \subset \{z\text{-axis}\}$. Assume moreover:

\begin{itemize}
\item[A)] $\Gamma$ lies in a fixed $2$-plane through the $z$-axis, i.e.\ $\Gamma(u) = (0,\dots,0,x_j(u),0,\dots,0,z(u))$ for some fixed index $j\in\{2,\dots,n-1\}$ and all $u$.
\end{itemize}

Let $\Gamma' = \Gamma|_{]0, 1[}$: the rotation of $\Gamma'$ in $\mathbf{R}^n$ about the $z$-axis generates a $C^1$ surface $V$ of dimension $2$ in $\mathbf{R}^n$.

\begin{theorem}\label{thm:2.1}
The submanifold $V$ of $\mathbf{R}^n$ constructed above is conformal to $\Delta^*$ with respect to the metric induced by $\mathbf{R}^n$.
\end{theorem}

\begin{proof}
We distinguish the cases: i) $\Gamma$ has finite length, ii) $\Gamma$ has infinite length.

i) We parametrize $V$ by $(u, \theta)$ where $u$ is the distance from $C$ measured along a generating curve $\Gamma$ and $\theta$ is the angle of rotation measured from a fixed generating curve: we may assume $0 < u < 1$. The coordinate curves are orthogonal, so the metric in these coordinates is diagonal; in particular the generators $u \to \chi(u, \overline{\theta})$ are parametrized by arc length, so the coefficient of $du^2$ in the metric is 1, while the coefficient of $d\theta^2$ is $g^2(u)$ with $g(u) = \text{distance in } \mathbf{R}^n \text{ between } \chi(u, \theta) \text{ and } \{z\text{-axis}\}$, since the curve $\theta \to \chi(\overline{u}, \theta)$ is a circle in $\mathbf{R}^n$ of radius $g(u)$.

Explicitly, let $e_1,\dots,e_n$ be the standard basis of $\mathbf R^n$. Under hypothesis A) one has $g(u)=|x_j(u)|$ and
\begin{equation*}
\chi(u,\theta)=g(u)\cos\theta\,e_j+g(u)\sin\theta\,e_1+z(u)\,e_n .
\end{equation*}
Thus $V$ is obtained by rotating the plane curve $u\mapsto(g(u),z(u))$ in the fixed $2$-plane spanned by $e_j$ and $e_1$, while $z$ is kept fixed. Differentiating with respect to $u$,
\begin{equation*}
\partial_u\chi(u,\theta)=g'(u)\cos\theta\,e_j+g'(u)\sin\theta\,e_1+z'(u)\,e_n,
\end{equation*}
hence, for every $\theta$,
\begin{align*}
|\partial_u\chi(u,\theta)|^2&=g'(u)^2\cos^2\theta+g'(u)^2\sin^2\theta+z'(u)^2=g'(u)^2+z'(u)^2\\
&=x_j'(u)^2+z'(u)^2=|\Gamma'(u)|^2=1,
\end{align*}
since $u$ is the arc-length parameter of $\Gamma$. Hypothesis A) enters here: the direction of $\Gamma$ transverse to the $z$-axis is the fixed vector $e_j$, so that $|\partial_u\chi|=1$ for all $(u,\theta)$.

The metric is therefore $ds^2 = du^2 + g^2(u)d\theta^2$. We have $g(u) > 0$ since $\Gamma$ does not meet the $z$-axis, and $g(u) \le u$ since the line segment is the curve minimizing the distance between a point and a line. It follows that
\begin{equation*}
\lim_{u \to 0} \int_u^1 \frac{1}{g(s)} ds \ge \lim_{u \to 0} \int_u^1 \frac{1}{s} ds = \infty
\end{equation*}
If we set $t(u) = \int_1^u \frac{1}{g(s)} ds$ for $u < 1$ we have $dt = \frac{du}{g(u)}$, $g^2 dt^2 = du^2$, and with this new parameter we obtain a conformal metric for $V$: $ds^2 = g^2(t)[dt^2 + d\theta^2]$, in other words $(t, \theta) \in {]-\infty, 0[} \times [0, 2\pi[$ are isothermal parameters for $V$.

The map $\varphi : (t, \theta) \to e^t(\cos \theta + i \sin \theta)$ is therefore a conformal map, injective and surjective when regarded as a map $\varphi : V \to \Delta^*$.

ii) Let $\chi_0$ be an arbitrary point on $\Gamma$ to which we assign the origin of the distance, i.e.\ $u = 0$, and let $\Gamma$ be oriented so that the boundary component $C$ corresponds to $u = +\infty$. The surface of revolution $V$ is again $C^1$-parametrized by the parameters $u$ and $\theta$, and again the metric is $ds^2 = du^2 + g^2(u)d\theta^2$.

Since $C \subset \{z\text{-axis}\}$, $\lim_{u \to \infty} g(u) = \lim_{u \to \infty} d(\chi(u, \theta), z\text{-axis}) = 0$ and in particular $g(u) \le 1$ for $u \ge M$ sufficiently large.

It follows that
\begin{equation*}
\rho(u) = \int_1^u \frac{1}{g(s)} ds \ge \int_1^M \frac{1}{g(s)} ds + (u - M)
\end{equation*}
and hence $\lim_{u \to \infty} \rho(u) = \infty$.

With parameters $(\rho, \theta)$ the metric $ds^2 = g^2(\rho)(d\rho^2 + d\theta^2)$ is conformal, and the conformal map $\varphi : {]0, \infty[} \times [0, 2\pi[ \to \Delta^*$ given by $\varphi(\rho, \theta) = e^{-\rho}(\cos \theta + i \sin \theta)$ induces a conformal map of $V$ onto $\Delta^*$.
\end{proof}

\section{Moduli of rings}

\noindent All rings considered below are domains in Riemann surfaces, hence of real dimension $2$.

\noindent In what follows we shall use the Lebesgue measure $\tau_S$ on a $\mathcal{C}^1$ submanifold of $\mathbb{R}^n$, uniquely determined by the following properties \cite{Su}:

\noindent a) $N$ open in $S \Rightarrow \tau_N(E) = \tau_S(E)$ for every Borel set $E \subset N$

\noindent b) if $f : S \to S^*$ is a $\mathcal{C}^1$ diffeomorphism of submanifolds of $\mathbb{R}^n$ then
\begin{equation*}
\tau_{S^*}(f(E)) = \int_E J_f d\tau_S
\end{equation*}

\noindent c) $S = \mathbb{R}^n \Rightarrow \tau_S$ is the ordinary Lebesgue measure.

\begin{definition}[Modulus of a family of curves \cite{C}]
Let $\Gamma$ be a family of locally rectifiable curves contained in a $\mathcal{C}^2$ submanifold of $\mathbb{R}^n$, and let $F(\Gamma)$ be the family of functions $\rho(x)$ defined $\forall x \in S$ such that:

\noindent 1) $\rho(x) \geq 0 \quad \forall x \in S$;

\noindent 2) $\rho(x)$ is Borel-measurable;

\noindent 3) $\int_\gamma \rho ds \geq 1 \quad \forall \gamma \in \Gamma$.

\noindent We define the modulus of $\Gamma$:
\begin{equation*}
M(\Gamma) = \inf_{\rho \in F(\Gamma)} \int_S \rho^2 d\tau \qquad (d\tau \text{ \ is the Lebesgue measure on } S)
\end{equation*}
\end{definition}

\begin{definition}[Modulus of a ring]
A ring is an open subset of a $\mathcal{C}^2$ submanifold of $\mathbb{R}^n$ conformal to a domain $A$ in $\mathbb{R}^2$ whose complement consists of one bounded component and one unbounded component. The boundary of $A$ has two components $F_0$ and $F_1$. Let $\Gamma_A$ be the family of locally rectifiable curves $\{\gamma \,|\, \gamma \cap F_i \neq \emptyset \text{ and } \gamma \cap (F_0 \cup F_1) = \text{endpoints of } \gamma\}$.

\noindent The modulus of a ring is $M(A) = 2\pi / M(\Gamma_A)$.\ \cite{AS}
\end{definition}

\begin{theorem}\label{thm:3.3}
For a circular ring $A = \{r_1 < |z| < r_2\}$, $M(A) = \log \left( \frac{r_2}{r_1} \right)$ \cite{AS}.
\end{theorem}

\begin{proof}
Let $\underline{e}$ be a unit vector and $\gamma_e$ a segment in $\Gamma_A$ parallel to $\underline{e}$, and let $\rho \in F(\Gamma_A)$. We observe that $ds = dr$ on $\gamma_e$ with $r = \text{distance from } (0,0)$
\begin{align*}
1 \leq \left( \int_{\gamma_e} \rho dr \right)^2 &= \left( \int_{\gamma_e} \rho r^{1/2} \cdot r^{-1/2} dr \right)^2 \leq \int_{r_1}^{r_2} \frac{1}{r} dr \cdot \int_{r_1}^{r_2} \rho^2 r dr \\
&= \log \frac{r_2}{r_1} \cdot \int_{r_1}^{r_2} \rho^2 r dr.
\end{align*}

\noindent Integrating over $\theta$ and applying Fubini's theorem
\begin{equation*}
2\pi \leq \log \frac{r_2}{r_1} \cdot \int_A \rho^2 dr
\end{equation*}

\noindent and hence $M(\Gamma_A) \geq 2\pi \left( \log \frac{r_2}{r_1} \right)^{-1}$.

\noindent Again integrating in polar coordinates, we observe that the function $\overline{\rho} = \frac{1}{r \log \frac{r_2}{r_1}}$ satisfies equality. We now prove a lemma, in some sense intuitive, which completes the proof by ensuring $\overline{\rho} \in F(\Gamma_A)$:

\begin{lemma}[\cite{V}]\label{lem:3.4}
Let $\rho$ be a nonnegative Borel-measurable function, defined on $[p,q]$, and let $\gamma : [0,\ell] \to \mathbb{R}^n$ be a continuous map whose image is rectifiable and such that $p \leq |\gamma(t)| \leq q \ \ \forall t$. Then
\begin{equation*}
\int_\gamma \rho(|x|) ds \geq \int_{|\gamma(0)|}^{|\gamma(\ell)|} \rho(r) dr.
\end{equation*}
\end{lemma}

\begin{proof}
We may assume that $r_1 = |\gamma(\ell)| \geq |\gamma(0)| = r_0$ and that $\gamma$ is the arc-length parametrization; $\gamma$ is thus $1$-Lipschitz: $|\gamma(s) - \gamma(s_0)| \leq |s - s_0|$. Setting $r(s) = |\gamma(s)|$ we have: $r(s)$ is A.C. $\Rightarrow r(s)$ is differentiable a.e., and for $\rho_k = \min(\rho, k)$ the change-of-variables formula holds:
\begin{equation*}
\int_{r_0}^{r_1} \rho_k(r) dr = \int_0^\ell \rho_k(r(s)) \frac{dr}{ds} ds = \int_\gamma \rho_k \frac{dr}{ds} ds.
\end{equation*}

\begin{equation*}
\left| \frac{dr}{ds} \right| \leq \lim_{s \to s_0} \frac{|\gamma(s_0) - \gamma(s)|}{|s - s_0|} \leq 1
\end{equation*}

\noindent and hence
\begin{equation*}
\int_\gamma \rho_k \frac{dr}{ds} ds \leq \int_\gamma \rho_k \left( \frac{dr}{ds} \right)^{\!\!+} ds \leq \int_\gamma \rho_k ds \leq \int_\gamma \rho ds
\end{equation*}

\noindent taking the limit $k \to \infty$, and using the Monotone Convergence Theorem on the left-hand side, we obtain
\begin{equation*}
\int_{r_0}^{r_1} \rho(r) dr \leq \int_\gamma \rho(|x|) ds \qedhere
\end{equation*}
\end{proof}

\noindent This shows $\overline\rho \in F(\Gamma_A)$, completing the proof.
\end{proof}

\begin{lemma}[\cite{AS}]\label{lem:3.5}
Let $A,\, A_i$, $i = 1, 2, \ldots$ be rings such that $A_i \subset A \quad \forall i$ and $A_i \cap A_j = \emptyset$. Then:
\begin{equation*}
M(A) \geq \sum_{i=1}^\infty M(A_i).
\end{equation*}
\end{lemma}

\begin{proof}
First we note that we may assume $M(A_i) < \infty \ \ \forall i$. Indeed, let $\rho \in F(\Gamma_{A_i})$: since every curve $\gamma$ in $\Gamma_A$ restricts to a curve $\tilde{\gamma}$ in $\Gamma_{A_i}$, then $\rho \in F(\Gamma_A)$: it follows that
\begin{equation*}
M(\Gamma_{A_i}) \geq M(\Gamma_A) \Rightarrow M(A) \geq M(A_i) \quad \forall i.
\end{equation*}

\noindent If $M(A_i) = \infty$ for some $i$ then $M(A) = \infty$ and the inequality holds. Take $\rho_i \in F(\Gamma_i)$ such that
\begin{equation*}
\int_{A_i} \rho_i^2 d\tau < (1 + \epsilon) M(\Gamma_{A_i}).
\end{equation*}

\noindent Then we have
\begin{equation*}
\frac{1}{\int_{A_i} \rho_i^2 d\tau} > \frac{(1 - \epsilon) M(A_i)}{2\pi}.
\end{equation*}

\noindent Fix $m \in \mathbf{N}$ and set
\begin{equation*}
\rho = \begin{cases}
\dfrac{M(A_i)}{\sum\limits_{k=1}^m M(A_k)} \rho_i & \text{on } A_i, \quad i \leq m \\[2ex]
0 & \text{on } A - \bigcup\limits_{k=1}^m A_k
\end{cases}
\end{equation*}

\noindent We have $\rho \in F(\Gamma_A)$ and

\begin{equation*}
\begin{gathered}
M(A) \geq \frac{\left(\sum\limits_{k=1}^m M(A_k)\right)^2}{\sum\limits_{k=1}^m M(A_k)^2 \int_{A_i} \rho_k^2 d\tau} \geq (1 - \epsilon) \frac{\left(\sum\limits_{k=1}^m M(A_k)\right)^2}{\sum\limits_{k=1}^m M(A_k)} = \\
= (1 - \epsilon) \sum_{k=1}^m M(A_k) \qquad \forall \epsilon \text{ and } m. \qedhere
\end{gathered}
\end{equation*}
\end{proof}

\begin{remark}[Conformal invariance]\label{rem:3.6}
Let $S$ be an orientable $\mathcal{C}^2$ submanifold of $\mathbf{R}^n$ and $\Gamma$ a family of locally rectifiable curves contained in an open set $\Omega \subset S$. Let $f$ be a direct or indirect conformal map from $\Omega$ onto an open set $\Omega_1$ of a $\mathcal{C}^2$ submanifold $S^*$.

\noindent Set $\Gamma^* = f(\Gamma) = \text{the family of images } f(\gamma) \text{ of } \gamma \in \Gamma$. Then
\begin{equation*}
M(\Gamma) = M(\Gamma^*).
\end{equation*}

\noindent For the proof, we may assume that we have found Riemann surface structures for $S, S^*$ compatible with the metric induced by $\mathbf{R}^n$, and let $z$ be a local complex coordinate for $S$. Let $f'$ be the complex derivative of $f$, or its conjugate if $f$ is indirect. Let $\rho^* \in F(\Gamma^*)$ and $\rho = \rho^*(f(z))|f'(z)|$. Then clearly $\rho \in F(\Gamma)$, and since $J_f = |f'|^2$, also $\int_S \rho d\tau = \int_{S^*} \rho^* d\tau^*$. It follows that $M(\Gamma) \leq M(\Gamma^*)$. The opposite inequality is obtained by considering $f^{-1}$.
\end{remark}

\section{The Whitney condition}\label{sec:whitney}

\noindent The definitions and results of this section hold for submanifolds of any dimension; from Section~\ref{sec:main} on, the dimension is $2$.

\begin{definition}\label{def:4.1}
Let $M_i$, $i = 1, 2$ be connected $\mathcal{C}^1$ manifolds in $\mathbf{R}^n$; with $\overline{M_1} \supset M_2$ and $M_1 \cap M_2 = \emptyset$. Given $(x, y) \in M_1 \times M_2$ let $\overline{xy}$ be the line joining $x$ and $y$ in $\mathbf{R}^n$. Consider, in place of the vector space $T_x M_1$ tangent to $M_1$ at $x$, the affine space that is its

\noindent translate by the vector $x$, still denoted $T_x M_1$. Let $\theta(T_x M_1, \overline{xy})$ be the angle between $\overline{xy}$ and its orthogonal projection onto $T_x M_1$, and define
\begin{equation*}
W(x, y) = \sin\left(\theta(T_x M_1, \overline{xy})\right)
\end{equation*}

\noindent $W(x, y)$ is called the Whitney function on $M_1 \times M_2$.\ \cite{W}
\end{definition}

\noindent Let $\tilde{M} = M_1 \cup M_2 \subseteq \mathbf{R}^n$ and let $\Delta$ be the diagonal of $M_2 \times M_2$, which is contained in $\tilde{M} \times M_2$. With reference to this notation we state:

\begin{definition}\label{def:4.2}
$(M_1, M_2)$ satisfies a Whitney condition in $\mathbf{R}^n$ if and only if $W$ extends to a continuous function $\tilde{W}$ on $\tilde{M} \times M_2$ and $\tilde{W}$ vanishes on $\Delta$.
\end{definition}

\begin{definition}\label{def:4.3}
Let $M_1, M_2$ and $W(x, y)$ be as above: we say that $(M_1, M_2)$ satisfies a strict Whitney condition in $\mathbf{R}^n$ if for every compact subset $K$ of $M_2$ there exist $\nu \in \mathbf{N}$ and $C \in \mathbf{R}_+$ such that
\begin{equation*}
W(x, y)^{\nu} \leq C |x - y|
\end{equation*}

\noindent for all $(x, y) \in M_1 \times K$.
\end{definition}

\stepcounter{theorem}

\noindent\textit{Example} \cite{H1}.
\emph{Let $M_1 \subseteq \mathbf{R}^2$ be described in polar coordinates by $M_1 = \{(r, \theta) : r = e^{-\theta^2}, 0 < \theta < \infty\}$ and $M_2 = \{0\}$; we have $M_1 \cap M_2 = \emptyset$ and $\overline{M_1} \supset M_2$.}

\noindent The Whitney function on $M_1 \times M_2$ is given by
\begin{equation*}
W(x, 0) = \frac{\left| \frac{dx_1}{d\theta} \cdot x_2 - \frac{dx_2}{d\theta} \cdot x_1 \right|}{\sqrt{\left(\frac{dx_1}{d\theta}\right)^2 + \left(\frac{dx_2}{d\theta}\right)^2} \sqrt{x_1^2 + x_2^2}} = \frac{1}{\sqrt{4\theta^2 + 1}}
\end{equation*}

\noindent which tends to zero as $\theta \to \infty$, while $\frac{W^{\nu}(x, 0)}{d(x, 0)} = \frac{e^{\theta^2}}{(4\theta^2 + 1)^{\nu/2}} \nearrow \infty$ for every $\nu\in\mathbf N$ (the exponential beats any power of $\theta$), so the strict Whitney condition fails.

\noindent The strict Whitney condition is not needed for the proof of our results as the theorem on point-like singularities  \ref{thm:5.3} and the lemma on metric decay of subanalytics \ref{lem:7.5}  and will only be mentioned in the counterexample of Section~\ref{sec:counter}.

\noindent Let now $(M_1, M_2)$ be a pair as above with $M_2 = \{0\}$. We shall denote by $d$ the distance function from the origin in $\mathbf{R}^n$ and by $B(0, r), S(0, r)$ the ball and the sphere in $\mathbf{R}^n$ of center 0 and radius $r$. When $X$ is a rectifiable submanifold of $\mathbf{R}^n$ with $\dim X = 1$, we shall denote by $L(X)$ the length of $X$.

\begin{lemma}\label{lem:4.5}
$(M, \{0\})$ satisfies a Whitney condition $\Rightarrow \exists$ a neighborhood $U$ of 0 in $\mathbf{R}^n$ such that $d|M$ has no critical points on $U \cap M$.
\end{lemma}

\begin{proof}
We observe that $d|M$ has a critical point $x \Leftrightarrow \operatorname{grad} d \perp T_x M \Leftrightarrow x \perp T_x M \Leftrightarrow W(x, 0) = 1$. $W$ is continuous on $M \cup \{0\}$ by hypothesis and $W(0, 0) = 0$, and this proves the lemma.
\end{proof}

\begin{corollary}[\cite{Hi}]\label{cor:4.6}
Let $(M, \{0\})$ be as above. Then $\exists r_0$ such that $S(0, r) \cap M$ are diffeomorphic $\mathcal{C}^1$ manifolds for $r \leq r_0$, of dimension $= \dim M - 1$.

\noindent Moreover $B(0, r_0) \cap M$ is diffeomorphic to $(S(0, r_0) \cap M) \times ]0, r_0[$ with a diffeomorphism $\psi$ such that $\psi(S(0, r) \cap M) = S(0, r_0) \cap M \times \{r\}$.
\end{corollary}

\begin{addendum}\label{add:4.7}
Let $(M,\{0\})$ satisfy a Whitney condition, let $r(x)=|x|$, and let $\nabla_M r(x)$ denote the component of $\operatorname{grad} r = x/|x|$ tangent to $M$ at $x$ (the gradient of $r|_M$). Then
\begin{equation*}
|\nabla_M r(x)| = \cos\bigl(\theta(T_xM,\overline{0x})\bigr) = \sqrt{1-W(x,0)^2},
\end{equation*}
where $W$ is the Whitney function of Definition \ref{def:4.1}. Consequently, since $W(x,0)\to W(0,0)=0$ as $x\to0$ (this is exactly the Whitney condition, Definition \ref{def:4.2}), there exists $r_1\le r_0$ such that
\begin{equation*}
|\nabla_M r(x)| \ge \tfrac12 \qquad \text{for all } x \in M\cap B(0,r_1).
\end{equation*}
\end{addendum}

\begin{proof}
$\operatorname{grad} r$ is a unit vector, and it decomposes orthogonally into its component tangent to $M$ at $x$, of length $\cos\theta(T_xM,\overline{0x})$, and its component normal to $M$ at $x$, of length $\sin\theta(T_xM,\overline{0x}) = W(x,0)$; this is precisely the geometric content of Definition \ref{def:4.1} applied to the pair $(M,\{0\})$, since $\overline{xy}$ with $y=0$ is the line through $x$ in the direction of $\operatorname{grad} r(x)$. Pythagoras' theorem for this orthogonal decomposition gives the stated identity. The last claim then follows from the continuity of $\tilde W$ on $\tilde M\times\{0\}$ and $\tilde W(0,0)=0$, by taking $r_1$ small enough that $W(x,0)\le \tfrac{\sqrt3}{2}$, hence $|\nabla_M r(x)|=\sqrt{1-W(x,0)^2}\ge\tfrac12$, for $x\in M\cap B(0,r_1)$.
\end{proof}

\section{Whitney and metric decay imply point-like singularities}\label{sec:main}

\begin{definition}\label{def:5.1}
Let 0 be a point of a $\mathcal{C}^0$ submanifold $E \subset \mathbf{R}^n$ such that $E - \{0\}$ is an orientable $\mathcal{C}^2$ manifold with $\dim E = 2$. $\{0\}$ is called a conformally point-like singularity if $\exists r$ such that $B(0, r) \cap (E - \{0\})$ is the union of connected components conformal to the unit disc minus the origin.
\end{definition}

\begin{definition}\label{def:5.2}

\noindent Let $E$ be as in Definition \ref{def:5.1}: $\{0\}$ is an $M$-regular singularity of $E$ if
\begin{enumerate}
\item[1)] The pair $(E - \{0\}, \{0\})$ satisfies a Whitney condition.
\item[2)] $\exists r_0$ such that $L(S(0, r) \cap E) \leq Cr$ for $r \leq r_0$
\end{enumerate}
\end{definition}

\begin{theorem}[On point-like singularities]\label{thm:5.3}
Every $M$-regular singularity (of a set $E$ with $\dim E=2$) is conformally point-like.
\end{theorem}

\begin{proof}
Let $W$ be a connected component of $B(0, r_0) \cap (E - \{0\})$. $W$ is conformal to $\Delta^*$ or to $A_r = \{r < |z| < 1, r > 0\}$. By Remark \ref{rem:3.6}, Lemma \ref{lem:3.5}, and Theorem \ref{thm:3.3} it suffices to show that there exists a sequence of disjoint rings $A_i$ in $W$ with
\begin{equation*}
\sum_{i=1}^{\infty} M(A_i) = \infty.
\end{equation*}
Set $A_i = \left( B\left(0, \frac{1}{2^{2i+1}}\right) - B\left(0, \frac{1}{2^{2i+2}}\right) \right) \cap W$

\noindent Since $(E - \{0\}, \{0\})$ satisfies a Whitney condition and by Corollary \ref{cor:4.6}, each $A_i$ is diffeomorphic (and hence conformal) to a ring.

\noindent Fixing $i$, we observe that the homothety of $\mathbf{R}^n$ $\psi : (x_1, ..., x_n) \to 2^{2i+2}(x_1, ..., x_n)$ is a conformal map, so $\psi(A_i) = A_i'$ is a ring with $M(A_i') = M(A_i)$. The boundary of $\psi(A_i)$ has components $C_1$ and $C_2$ contained respectively in the spheres $S(0, 1)$ and $S(0, 2)$. Let $\bar{\rho}$ be the function constantly equal to 1 on $A_i'$ and $\gamma \in \Gamma_{A_i'}$; by Lemma \ref{lem:3.4}
\begin{align*}
\int_{\gamma} 1 \, ds \geq \int_{1}^{2} 1 \, dr = 1 &\Rightarrow \bar{\rho} \in F(\Gamma_{A_i'}) \\
\Rightarrow M(\Gamma_{A_i'}) \leq \int_{A_i'} 1 \, d\tau &= \text{Area } (A_i')
\end{align*}

\noindent Put $a_i=2^{-(2i+2)}$, $b_i=2^{-(2i+1)}$, so that $A_i=\{x\in W:\ a_i<|x|<b_i\}$ up to the spheres $S(0,a_i)$, $S(0,b_i)$, which meet $W$ in sets of area zero. Let $r(x)=|x|$. By Corollary \ref{cor:4.6}, $S(0,s)\cap W$ is a $\mathcal C^1$ curve for $s\le r_0$, and $L(s):=L(S(0,s)\cap W)\le L(S(0,s)\cap E)\le Cs$ by condition 2) of Definition \ref{def:5.2}. By Addendum \ref{add:4.7} there is $r_1$ such that $|\nabla_W r|\ge\tfrac12$ on $W\cap B(0,r_1)$; we consider only the indices $i\ge i_0$ with $b_{i_0}<\min(r_0,r_1)$, which does not affect the divergence of the series.

\noindent The function $r|_W$ is Lipschitz on the $2$-manifold $W$, and $|\nabla_W r|\ne0$ on $A_i$. The coarea formula \cite[Thm.~3.2.12]{Fed} applied to $r|_W$ and to the integrand $1/|\nabla_W r|$ on $A_i$ gives
\begin{align*}
\text{Area}(A_i)&=\int_{A_i}\frac{|\nabla_W r|}{|\nabla_W r|}\,d\mathcal H^2=\int_{a_i}^{b_i}\int_{A_i\cap\{r=s\}}\frac{1}{|\nabla_W r|}\,d\mathcal H^1\,ds\\
&\le 2\int_{a_i}^{b_i}L(s)\,ds\le 2C\int_{a_i}^{b_i}s\,ds=C\,(b_i^2-a_i^2).
\end{align*}
Here $A_i\cap\{r=s\}=S(0,s)\cap W$ for $a_i<s<b_i$. Since $b_i^2-a_i^2=2^{-(4i+2)}-2^{-(4i+4)}=3\cdot2^{-(4i+4)}$, the homothety $\psi$ (which multiplies areas by $2^{4i+4}$) gives
\begin{equation*}
\text{Area}(A_i')=2^{4i+4}\,\text{Area}(A_i)\le 3C,
\end{equation*}
hence $M(\Gamma_{A_i'})\le3C$ and
\begin{equation*}
M(A_i)=M(A_i')=\frac{2\pi}{M(\Gamma_{A_i'})}\ge\frac{2\pi}{3C}\qquad\forall\, i\ge i_0 \ \Rightarrow\ \sum_{i\ge i_0}M(A_i)=\infty.
\end{equation*}

\noindent By Section \ref{sec:intro}, $W$ is conformal to $\Delta^*$, to $\mathbf{C}^*$, or to $A_\rho=\{\rho<|z|<1\}$ for some $\rho\in(0,1)$, and $\mathbf{C}^*$ has already been excluded there. Suppose, for contradiction, that $W$ is conformal to $A_\rho$ for some $\rho\in(0,1)$. The rings $A_i$, $i\ge i_0$, are pairwise disjoint subsets of $W$, nested in order of decreasing $i$ toward the puncture $\{0\}$; every curve of $\Gamma_W$ joining the two ends of $W$ therefore restricts to a curve of $\Gamma_{A_i}$ for every $i$, so that Lemma \ref{lem:3.5} applies with $A=W$ and gives
\begin{equation*}
M(W) \geq \sum_{i\ge i_0} M(A_i) = \infty.
\end{equation*}
On the other hand, by Remark \ref{rem:3.6} (conformal invariance of the modulus, applied to the conformal map $W\to A_\rho$) and Theorem \ref{thm:3.3},
\begin{equation*}
M(W) = M(A_\rho) = \log\frac{1}{\rho} < \infty,
\end{equation*}
a contradiction. Hence $W$ is not conformal to any $A_\rho$, and therefore $W$ is conformal to $\Delta^*$.
\end{proof}

\section{Strict Whitney and metric decay are independent}\label{sec:counter}

\noindent In this section we shall construct a surface $S$ in $\mathbf{R}^3$ with an isolated singularity satisfying even a strict Whitney condition \ref{def:4.3}  at the singular point, but such that $S(0, r) \cap S$ has length tending to $\infty$ as $r \to 0$.

\subsection{Basic construction}

\noindent The surface $S$ that we shall construct will be $C^\infty$ away from the singular point $\{0\}$ and will be modeled, in successive ``strips'', on the following sequence of geometric figures of the Koch snowflake curve:

\begin{center}
\begin{tikzpicture}[scale=2.1]
  \begin{scope}[xshift=0cm]
    \draw[thick] (0,1.5) -- (0.866,0) -- (-0.866,0) -- cycle;
   \node at (0.25,-0.8) {\footnotesize $\Gamma_1$};
  \end{scope}

  \begin{scope}[xshift=2.6cm]
 \draw[thick, decoration=Koch snowflake] decorate {
      (0,1.5) -- (0.866,0) -- (-0.866,0) -- cycle
    };
    \node at (0.25,-0.8) {\footnotesize $\Gamma_2$};
  \end{scope}

  \begin{scope}[xshift=5.2cm]
    \draw[thick, decoration=Koch snowflake] decorate { decorate {
      (0,1.5) -- (0.866,0) -- (-0.866,0) -- cycle
    } };
   \node at (0.25,-0.8) {\footnotesize $\Gamma_3$};
  \end{scope}
\end{tikzpicture}
\end{center}

 \noindent In each figure all angles are $60^\circ$ and each figure is obtained from the previous one by replacing each segment of length $\ell$ with a broken line of length $\frac{4}{3} \cdot \ell$ in the following way

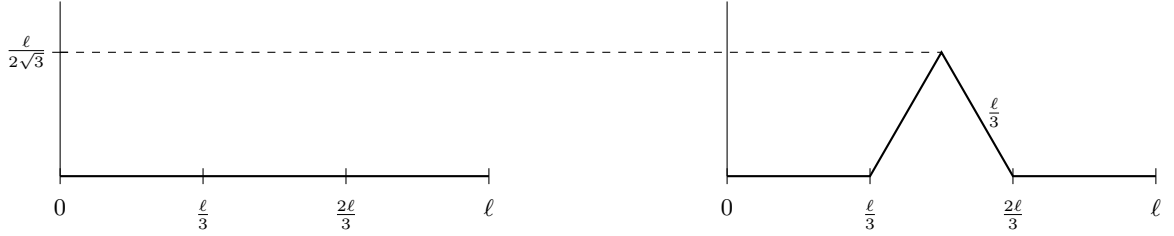
\begin{center}
\begin{tikzpicture}[scale=2.1]
  \draw[thick] (0,0) -- (2.7,0);
  \draw (0,0.05) -- (0,-0.05) node[below=2pt] {\footnotesize $0$};
  \draw (0.9,0.05) -- (0.9,-0.05) node[below=2pt] {\footnotesize $\frac{\ell}{3}$};
  \draw (1.8,0.05) -- (1.8,-0.05) node[below=2pt] {\footnotesize $\frac{2\ell}{3}$};
  \draw (2.7,0.05) -- (2.7,-0.05) node[below=2pt] {\footnotesize $\ell$};
  
  \draw (0,0) -- (0,1.1);
  \draw (-0.05, 0.78) -- (0.05, 0.78);
  \node[left=1pt] at (0, 0.78) {\footnotesize $\frac{\ell}{2\sqrt{3}}$};
  
  \draw[dashed] (0,0.78) -- (5.55, 0.78);

  \begin{scope}[xshift=4.2cm]
    \draw (0,0) -- (0,1.1);
    \draw[thick] (0,0) -- (0.9,0) -- (1.35,0.78) -- (1.8,0) -- (2.7,0);
    
    \draw (0,0.05) -- (0,-0.05) node[below=2pt] {\footnotesize $0$};
    \draw (0.9,0.05) -- (0.9,-0.05) node[below=2pt] {\footnotesize $\frac{\ell}{3}$};
    \draw (1.8,0.05) -- (1.8,-0.05) node[below=2pt] {\footnotesize $\frac{2\ell}{3}$};
    \draw (2.7,0.05) -- (2.7,-0.05) node[below=2pt] {\footnotesize $\ell$};
    
    \node[right=1pt] at (1.55, 0.4) {\footnotesize $\frac{\ell}{3}$};
  \end{scope}
\end{tikzpicture}
\end{center}

\noindent Let $L(\Gamma_n)$ be the length of $\Gamma_n$; we have $L(\Gamma_n) = \frac{4}{3} L(\Gamma_{n-1}) = \cdots = (\frac{4}{3})^{n-1} L(\Gamma_1)$.

\noindent In particular $L(\Gamma_n) \nearrow \infty$, while it could be shown that if $A(\Gamma_n)$ is the area enclosed by the broken line $\Gamma_n$, $\exists \lim_{n \to \infty} A(\Gamma_n) < \infty$

\noindent We shall approximate each $\Gamma_n$ by a $C^\infty$ curve $C_n$ with $L(C_n) \geq \frac{L(\Gamma_n)}{3}$. Each curve $C_n$ is homotopic to $C_{n+1}$, and we may view the homotopy as a strip, that is, a $C^\infty$ surface in $\mathbf{R}^3$ with boundary $C_n \cup C_{n+1}$. By suitably joining these strips (cobordisms) and collapsing to $\{0\}$ the ``end'' corresponding to $n \to \infty$, we obtain a surface in $\mathbf{R}^3$ with $\{0\}$ as an isolated singular point. To conclude, we shall further modify the surface by projecting it onto concentric spheres centered at $\{0\}$.

\subsection{Approximation of the $\Gamma_n$}

\noindent Let $\ell_n =$ the length of each segment of $\Gamma_n$. From a fixed $C^\infty$ step function:
\begin{equation*}
\rho' : \mathbf{R} \to \left[ \frac{-\sqrt{3}}{2}, \frac{\sqrt{3}}{2} \right]
\end{equation*}
\begin{equation*}
\rho'(-1) = +\frac{\sqrt{3}}{2} \quad \rho(1) = -\frac{\sqrt{3}}{2}
\end{equation*}

\noindent we obtain the function
\begin{equation*}
\rho_n'(x) = \rho'\left( \frac{12}{\ell_n} \cdot x \right)
\end{equation*}
\begin{equation*}
\rho_n'(x) = \begin{cases}
+\sqrt{\frac{3}{2}} & \text{for } x \leq -\frac{\ell_n}{12} \\
-\sqrt{\frac{3}{2}} & \text{for } x \geq +\frac{\ell_n}{12}
\end{cases}
\end{equation*}
and then, by integration, the function $\rho_n(x)$, with the following graph

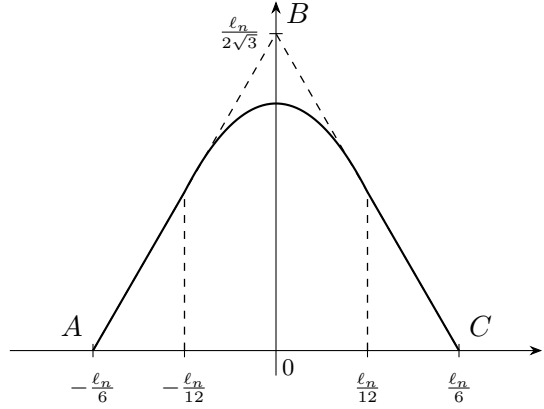
\begin{center}
\begin{tikzpicture}[scale=1.1, >=Stealth]
  \draw[->] (-3.2, 0) -- (3.2, 0);
  \draw[->] (0, -0.3) -- (0, 4.2);
  
  \draw[dashed, line width=0.5pt] (-2.2, 0) -- (0, 3.81) -- (2.2, 0);
  
  \draw[thick] (-2.2, 0) -- (-1.1, 1.9055) plot[domain=-1.1:1.1, samples=50] (\x, {2.9703 - 0.88*\x*\x}) -- (2.2, 0);

  \draw[dashed, line width=0.5pt] (-1.1, 0) -- (-1.1, 1.9055);
  \draw[dashed, line width=0.5pt] (1.1, 0) -- (1.1, 1.9055);

  \draw (-2.2, 0.08) -- (-2.2, -0.08) node[below=2pt] {\footnotesize $-\frac{\ell_n}{6}$};
  \draw (-1.1, 0.08) -- (-1.1, -0.08) node[below=2pt] {\footnotesize $-\frac{\ell_n}{12}$};
  \node at (0.15, -0.2) {\footnotesize $0$};
  \draw (1.1, 0.08) -- (1.1, -0.08) node[below=2pt] {\footnotesize $\frac{\ell_n}{12}$};
  \draw (2.2, 0.08) -- (2.2, -0.08) node[below=2pt] {\footnotesize $\frac{\ell_n}{6}$};
  
  \draw (-0.08, 3.810) -- (0.08, 3.81);
  \node[left=2pt] at (0, 3.81) {\footnotesize $\frac{\ell_n}{2\sqrt{3}}$};
  
  \node[above left] at (-2.2, 0.05) {$A$};
  \node[above right] at (0, 3.81) {$B$};
  \node[above right] at (2.2, 0.05) {$C$};
\end{tikzpicture}
\end{center}

\noindent $\rho_n$ approximates the broken line $\Gamma_n$ near a $60^\circ$ angle.

\noindent We have $L(ABC) \geq L(AC) \geq L(\overline{AC}) = \frac{\ell_n}{3}$

\noindent Each $\Gamma_n$ is approximated by a curve $C_n$ with $L(C_n) > \frac{1}{3} L(\Gamma_n)$, since $C_n$ certainly contains at least $\frac{1}{3}$ of each segment of $\Gamma_n$.

\subsection{Homotopy between $C_n$ and $C_{n+1}$}

\noindent We shall construct the homotopy ``piecewise'' for each segment and its corresponding one, so that it is the identity on the intersections of these pieces, in such a way that all the local homotopies join together $C^\infty$: if $f_n$ and $f_{n+1}$ are the functions whose graphs are the arcs $\beta_n$ and $\beta_{n+1}$, we set $h_n(x, t) = t(f_{(n+1)}(x) - f_n(x))$.

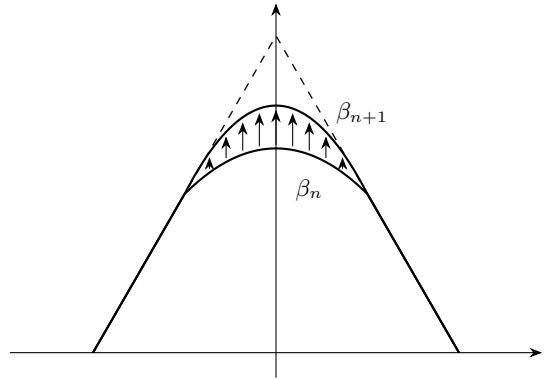
\begin{center}
\begin{tikzpicture}[scale=1.1, >=Stealth]
  \draw[->] (-3.2, 0) -- (3.2, 0);
  \draw[->] (0, -0.3) -- (0, 4.2);
  
  \draw[dashed, line width=0.5pt] (-2.2, 0) -- (0, 3.81) -- (2.2, 0);
  
  \draw[thick] (-2.2, 0) -- (-1.1, 1.9055) plot[domain=-1.1:1.1, samples=50] (\x, {2.9703 - 0.88*\x*\x}) -- (2.2, 0);

  \draw[thick] (-2.2, 0) -- (-1.1, 1.9055) plot[domain=-1.1:1.1, samples=50] (\x, {2.4554 - 0.4545*\x*\x}) -- (2.2, 0);

  \draw[->] (-0.8, 2.21) -- (-0.8, 2.35);
  \draw[->] (-0.6, 2.34) -- (-0.6, 2.6);
  \draw[->] (-0.4, 2.43) -- (-0.4, 2.77);
  \draw[->] (-0.2, 2.48) -- (-0.2, 2.88);
  \draw[->] (0.0, 2.50) -- (0.0, 2.92);
 \draw[->] (0.8, 2.21) -- (0.8, 2.35);
  \draw[->] (0.6, 2.34) -- (0.6, 2.6);
  \draw[->] (0.4, 2.43) -- (0.4, 2.77);
  \draw[->] (0.2, 2.48) -- (0.2, 2.88);
 
  \node[right] at (0.6, 2.88) {\footnotesize $\beta_{n+1}$};
  \node[below] at (0.4, 2.21) {\footnotesize $\beta_n$};
 
\end{tikzpicture}
\end{center}

\subsection{Conditions on the tangent spaces}

\noindent Denoting generally by $f_1, f_0$ the functions whose graphs are the arcs in question, we can write the homotopy in $\mathbb{R}^3$ as the surface

\begin{equation*}
g(x, t) = (x, t, t(f_1(x) - f_0(x)) + f_0(x)) \quad 0 \le t \le 1 \quad -a \le x \le a
\end{equation*}

The plane tangent to this surface at points with coordinates $(0, t)$ $0 < t < 1$ is spanned by the vectors $(1, 0, 0)$ and $(0, 1, f_1(0) - f_0(0))$; the projection of the vector $(0, 1, 0)$ onto the tangent plane at these points is therefore parallel to $(0, 1, f_1 - f_0)$, and the angle between the vector and its projection has cosine $\frac{1}{\sqrt{1+(f_1-f_0)^2}}$.

In the general case of points with coordinates $(x, t)$, $x \neq 0$, this is no longer exact, but it is true that the angle between $(0, 1, 0)$ and its projection onto the tangent space, an angle which we generally denote by $\theta(x, t)$, can be bounded above by the angle between $(0, 1, 0)$ and $(0, 1, f_1(x) - f_0(x))$. It follows that $\cos\,\theta(x, t) \ge \frac{1}{\sqrt{1+(f_1-f_0)^2}}$. Suppose the homotopy is between $C_n$ and $C_{n+1}$: then $f_1(x) - f_0(x) \le \frac{\ell_{n+1}}{2\sqrt{3}}$ and hence

\begin{equation*}
\cos^2\,\theta(x, t) \ge 1 - \frac{\ell_{n+1}^2}{12} \Rightarrow \sin^2\,\theta(x, t) \le \frac{\ell_{n+1}^2}{12}
\end{equation*}

\subsection{Gluing the homotopies}

For the homotopy between $C_n$ and $C_{n+1}$ we set

\begin{equation*}
G(x, t) = g(x, \rho(t)) : x \in C_n, \quad t \in [0, 1] \quad \text{and}
\end{equation*}

\begin{equation*}
F(x, t) = G(x, (n+1)^2 t) \quad x \in C_n \quad t \in [0, 1/(n+1)^2]
\end{equation*}

with $\rho$ the usual $C^\infty$ step function: $\rho(0) = 0$, $\rho(1) = 1$: it is immediate to check that

\begin{equation*}
\cos\,\theta \ge \frac{1}{\sqrt{1 + (n+1)^4 (\dot{\rho})^2 (f_1 - f_0)^2}}
\end{equation*}

\begin{equation*}
\Rightarrow \sin^2\,\theta \le (n+1)^4 \dot{\rho}^2 \frac{\ell_n^2}{12} \le (n+1)^4 (\max \dot{\rho})^2 \frac{\ell_n^2}{12} = C(n+1)^4 \ell_n^2.
\end{equation*}

Let $H = \sum_{i=1}^\infty \frac{1}{i^2} < \infty$ and denote by $V_n$ the surface obtained from the homotopy between $C_n$, lying in the plane $\left\{ z = H - \sum_{i=1}^n \frac{1}{i^2} \right\}$, and $C_{n+1}$, lying in the plane $\left\{ z = H - \sum_{i=1}^{n+1} \frac{1}{i^2} \right\}$, both centered on the $z$-axis at their barycenter. $V_n$ and $V_{n+1}$ join together $C^\infty$ (owing to the use of step functions) and

\begin{equation*}
V = \bigcup_{n=1}^{\infty} V_n \cap \{0 < z < H\} \quad \text{is a} \quad C^\infty \text{ surface.}
\end{equation*}

Consider the function $z$ on $V_n$: since $1/i^2$ is decreasing, comparison with $\int_n^\infty dx/x^2=1/n$ and $\int_{n+1}^\infty dx/x^2=1/(n+1)$ gives

\begin{equation*}
\frac{1}{n+1} \ \le\ z \le H - \sum_{i=1}^{n} \frac{1}{i^2} = \sum_{n+1}^{\infty} \frac{1}{i^2} \ \le\ \frac{1}{n} \Rightarrow (1+n) \le \frac{2}{z}.
\end{equation*}

Similarly $(1+n) \ge \frac{1}{z}$. We can thus rewrite the bound for $P \in V_n$, $\theta(P)$ being the angle between the tangent plane at $P$ and the vector $(0, 0, -1)$

\begin{align*}
|\sin\,\theta(P)| &\le C(n+1)^4 \ell_{n+1}^2 = C' \frac{(n+1)^4}{9^{\,n+1}} \le C'\Bigl(\frac{2}{z}\Bigr)^{\!4} 9^{-1/z} \\
&= \frac{16C'}{z^4}\,9^{-1/z} \le k\,z^m \quad \forall m \in \mathbb{N} \quad 0 \le z \le \epsilon(m),
\end{align*}

using $(1+n)\le2/z$ in the polynomial factor and $(1+n)\ge1/z$ (so $9^{-(n+1)}\le9^{-1/z}$) in the exponential factor; the last inequality holds because $9^{-1/z}$ decays, as $z\to0^+$, faster than any power of $z$ can grow.

\subsection{Transformation into a surface with an isolated singularity $\{0\}$}

Consider the map $\psi : \mathbb{R}^3 \to \mathbb{R}^3 \quad \psi(x, y, z) = (xz, yz, z)$

$\psi$ is $C^\infty$ and invertible away from the plane $\{z = 0\}$, so $\psi(V) = W'$ is a $C^\infty$ surface, and we observe that $\{0\}$ is a point adherent to $W'$ and $W = W' \cup \{0\}$ is a surface with an isolated singularity at $\{0\}$. From the construction of $V$ we note that $x^2 + y^2 \le M \quad \forall (x, y, z) \in V$.

Let us now verify that the pair $(W', \{0\})$ satisfies a strict Whitney property. Let $Q$ be a point of $W'$ and $T_Q W'$ the plane tangent at $Q$ to $W'$. We must estimate the angle $\Omega(Q)$ between $-\frac{Q}{|Q|}$ and $T_Q W'$, which we can bound above by the angle $\Delta$ between $-\frac{Q}{|Q|}$ (or a scalar multiple $w$ of it) and any vector $w' \in T_Q W'$. If $Q = (x', y', z') = (xz, yz, z)$, set $w = -(x, y, 1)$. Let $v$ be a unit vector tangent at $P$ to $V$ with $\psi(P) = Q$, such that $v$ is a positive scalar multiple of the projection of $(0, 0, -1)$ onto $T_P V$.

We have already seen in the previous paragraphs that if $\theta(P)$ is the angle between $(0, 0, -1)$ and $v$, then $|\sin\,\theta(P)| \le k\,z^m$ for any integer $m$ and $z \le \epsilon(m)$, and hence $|(0, 0, -1) - v| = 2|\sin\,\frac{\theta}{2}| \le 2\,\sin\,\theta \le k' z^m$: set $w' = d\psi(v)$ and note that we have

\begin{equation*}
w = d\psi(0, 0, -1) = -(x, y, 1) \text{ If } a = w - \frac{\langle w, w'\rangle w'}{|w'|^2}
\end{equation*}
\begin{equation*}
\sin\,\Delta = \frac{|a|}{|w|}
\end{equation*}

recall $|w| \ge 1 \Rightarrow \sin\,\Delta \le |a|$

\begin{align*}
|\sin\,\Omega(Q)| &\le |\sin\,\Delta| = |a| \le |w - w'| = \\
&= |d\psi(0, 0, -1) - d\psi(v)| \le |d\psi| |(0, 0, -1) - v| \le \\
&\le \left(\sqrt{2z^2 + x^2 + y^2 + 1}\right) \cdot k' z^m \le \\
&\le \left(\sqrt{H + M + 1}\right) k' z^m
\end{align*}

since $z' = z$ we obtain $\forall m \in \mathbb{N},\, \exists \epsilon$ such that

\begin{equation*}
|\sin\,\Omega(x', y', z')| \le K(z')^m \quad \forall z \in [0, \epsilon]
\end{equation*}

$W$ resembles roughly a cone over $C_1$ with many ``bumps'' of decreasing size and increasing frequency near the vertex. We now denote by $D_r = W \cap \{z = r\}$. For simplicity set $r_n = H - \sum_{i=1}^n \frac{1}{i^2}$, and with this convention $D_n = D_{r_n}$. $D_r$ is obtained from the corresponding $C_r$ by a homothety of factor $r$ (see the transformation $V \to W$). In particular for the length of $D_r$ we have $L(D_r) = r L(C_r) > \frac{r L(\Gamma_r)}{3}$.

If $r_{n+1} \le r \le r_n$ then $L(D_r) \ge \frac{r_{n+1} L(\Gamma_n)}{3} > \frac{\ell_1}{n+1} \frac{4^n}{3^{n-1}}$, and hence $\lim_{r \to \infty} L(D_r) = \infty$.

\subsection{A strict Whitney $\mathcal{C}^\infty$ surface with unbounded length spherical slices}

Consider the map that orthogonally projects the points with $z > 0$ of the cylinder with axis $z$ and radius $r$ onto the upper hemisphere of radius $r$ centered at the origin.

If we restrict this map to the part of the plane $\{z = r\}$ contained in the cylinder, we obtain a diffeomorphism between a disc and a hemisphere. Combining all these maps as $r$ varies, we obtain

\begin{equation*}
\varphi(x, y, z) = \left(x, y, \sqrt{z^2 - x^2 - y^2}\right)
\end{equation*}

$\varphi$ is defined on the cone $z^2 \ge x^2 + y^2$. We may assume, after suitably shrinking $D_1$, that $W$ is contained in the cone $z^2 \ge k(x^2 + y^2)$, $k > 1$, contained in the cone where $\varphi$ is defined, and thus find the new $\mathcal{C}^\infty$ surface $S' = \varphi(W')$ and $S = \varphi(W)$; $S$ has an isolated singularity at $\{0\}$.

The derivatives of $\varphi$ on the cone $z^2 \ge k(x^2 + y^2)$, $k > 1$, are bounded: for example
\begin{align*}
\left| \frac{x}{(z^2 - x^2 - y^2)^{1/2}} \right| &\le \frac{|x|}{(k(x^2 + y^2) - x^2 - y^2)^{1/2}} \le \\[8pt]
&\le \frac{|x|}{(x^2 + y^2)^{1/2}} \frac{1}{k - 1} \le \frac{1}{k - 1}.
\end{align*}
Hence $W$ satisfies a strict Whitney condition $\Rightarrow S$ satisfies a strict Whitney condition. Let $\theta \to (x(r, \theta), y(r, \theta), z(r, \theta))$ be a parametrization of $D_r$, that is

$z(r, \theta) \equiv r$, $\theta \in [0, 2\pi[$. The curve that is the image of $D_r$ under $\varphi$ has parametrization $(x(\theta)$,

$y(\theta), \sqrt{r^2 - x^2(\theta) - y^2(\theta)})$, and by construction it is exactly $S \cap S(0, r)$, where $S(0, r)$ is the sphere centered at $\{0\}$ of radius $r$; $\varphi(D_r)$ therefore has length
\begin{align*}
&\int_0^{2\pi} \sqrt{\dot{x}^2 + \dot{y}^2 + \left[ \frac{d}{d\theta} \sqrt{z^2 - x^2 - y^2} \right]^2} \, d\theta > \\[8pt]
&> \int_0^{2\pi} \sqrt{\dot{x}^2 + \dot{y}^2} \, d\theta = L(D_r)
\end{align*}
Hence $L(\varphi(D_r)) \to \infty$ as $r \to 0$.

The surface $S$ is $C^\infty$ away from $\{0\}$ : both $\psi$ and $\varphi$ are diffeomorphisms on the open regions where they are used -- $\psi$ on $\{z\ne0\}$, $\varphi$ on the open cone $z^2>x^2+y^2$ -- and are injective on $V$, resp.\ $W'$, since distinct points there have distinct $z$-coordinates or lie in disjoint height ranges, so $S-\{0\}$ is a genuinely embedded $C^\infty$ surface, not merely immersed.
 $(S,\{0\})$ also satisfies a strict Whitney condition (Definition \ref{def:4.3}), in fact the stronger property $|\sin\theta(P)|\le kz^m$ for \emph{every} $m$; and $L(S(0,r)\cap S)\to\infty$ as $r\to0$. This example establishes the claim that a strict Whitney condition does not imply condition 2) of $M$-regularity (Definition \ref{def:5.2}), confirming these hypotheses are independent, as announced in the introduction.

\subsection{The relation to Whitney stratifications and conically smooth structures}

We show here that $\{S-\{0\},\{0\}\}$ form a Whitney stratification in the classical sense of Mather \cite{Ma} and Thom \cite{Th}.  Recall the classical requirements for a stratification of a set $X$ into locally closed smooth submanifolds (``strata''): (i) local finiteness; (ii) the frontier condition (the closure of each stratum is a union of strata); and, for each pair of incident strata $Y\subset\overline X{}'$ with $Y$ in the closure of a stratum $X'$, Whitney's conditions (a) and (b) at every point $y\in Y$: for sequences $x_i\in X'$, $x_i\to y$, with $T_{x_i}X'\to\tau$ (a limit plane, which exists along a subsequence by compactness of the Grassmannian), (a) requires $T_yY\subseteq\tau$, and (b) additionally requires that if $y_i\in Y$, $y_i\to y$ too, then the limiting secant direction of $\overline{x_iy_i}$ (again along a subsequence) also lies in $\tau$.

Here $X=S$, the two strata are $X'=S-\{0\}$ and $Y=\{0\}$: (i) and (ii) are immediate (finitely many -- indeed two -- strata; $\overline{S-\{0\}}=S=(S-\{0\})\cup\{0\}$, a union of strata; $\overline{\{0\}}=\{0\}$). Since $Y=\{0\}$ is $0$-dimensional, $T_yY=\{0\}$, so condition (a) holds trivially for \emph{any} limit plane $\tau$. For (b), since $Y$ is a single point, $y_i\equiv0$ is forced, so the secant $\overline{x_i0}$ is exactly the radial line through $x_i$, and the required condition -- that this line's limiting direction lies in $\tau=\lim T_{x_i}(S-\{0\})$ -- is precisely $W(x_i,0)\to0$ in the sense of Definition \ref{def:4.1}, i.e.\ exactly this paper's own Whitney condition (Definition \ref{def:4.2}) for the pair $(S-\{0\},\{0\})$. We have verified above that this holds -- indeed with the strong, all-$m$ polynomial (in fact exponential) rate -- so condition (b) holds, and with it (b)$\Rightarrow$(a) automatically in general (a classical fact, see \cite{Ma}, though here (a) was already immediate). Therefore $\{S-\{0\},\{0\}\}$ is a bona fide two-stratum Whitney stratification in the sense of \cite{Ma,Th}, and in fact satisfies conditions considerably stronger than the minimum Whitney (a)+(b) requires, since the angle bound is flat (beats every polynomial rate, not merely linear as condition (b) alone would need).

We note that the Whitney condition of Definition \ref{def:4.2} is, for a pair $(M,\{0\})$ with $\{0\}$ a single point, exactly a restatement of classical condition (b), of which [W] (Whitney's own 1965 paper) is a common ancestor together with \cite{Ma,Th}. This section, and Definition \ref{def:4.2} more broadly, may be regarded as a self-contained account of the relevant special case of Whitney-stratification theory, distinguishing it (via this very counterexample) from the different, metric ``decay'' condition 2) of Definition \ref{def:5.2} -- a distinction not drawn in these terms in \cite{Ma,Th} themselves, which are concerned with topological/isotopy consequences of Whitney's conditions rather than with metric decay of slice lengths.

Nocera and Volpe \cite{NV} recently proved that every Whitney stratified space of finite dimension -- in the same sense used above -- admits an essentially unique conically smooth structure in the sense of Ayala--Francis--Tanaka. Since $(S,\{0\})$ is a genuine two-stratum Whitney stratification, one might ask whether its unusual feature -- an isolated point whose ``link'' $D_r$ has length $\to\infty$ as $r\to0$ -- obstructs the existence of the conically smooth structure that \cite{NV} guarantees, making $S$ a counterexample. Of course it does not, and $S$ is not a counterexample.

\section{Subanalytic surfaces have point-like isolated singularities}\label{sec:subanalytic}

\begin{definition}[\cite{HS,H1}]\label{def:7.1}
Let $M, N$ be real analytic manifolds. A subset $E$ of $M$ is called subanalytic in $M$ if every point $x$ of $M$ has a neighborhood $U$ such that $E \cap U$ is the projection of a semianalytic set $A$ relatively compact in $M \times N$: $E \cap U = \prod(A)$ where $\prod : M \times N \to M$ is the natural projection.
\end{definition}

\begin{definition}[\cite{S}]\label{def:7.2}
A relatively compact subanalytic set is called an ``$L$-analytic piece'' if it has a presentation of the form $E = \{(u, \varphi(u)) : u \in \Omega\}$, with $\Omega$ an open set in the subspace $U$ of $X$, $X = U \oplus V$ with $V$ the orthogonal complement of $U$, and $\varphi : \Omega \to V$ an analytic function on $\Omega$ with first derivatives uniformly bounded on $\Omega$.
\end{definition}

\begin{theorem}[\cite{S}]\label{thm:7.3}
A relatively compact subanalytic set $E$ is a finite union of $L$-analytic pieces.
\end{theorem}

\begin{theorem}[\cite{H2}]\label{thm:7.4}
Let $M$ be a connected real analytic submanifold of $\mathbb{R}^n$, subanalytic in $\mathbb{R}^n$. Let $N \subseteq \overline{M} - M$ be subanalytic in $\mathbb{R}^n$ and locally closed. Then there exists a closed subset $S$ of $N$ such that
\begin{enumerate}
\item[1)] $S$ is subanalytic
\item[2)] $N - S$ is dense in $N$
\item[3)] $(M, N)$ satisfies a Whitney condition at every point of $N - S$.
\end{enumerate}
\end{theorem}

\begin{lemma}[Metric decay of subanalytic sets]\label{lem:7.5}
Let $E$ be subanalytic in $\mathbb{R}^n$ with $\dim E = p\ge1$: then $\text{Volume } (S(0, r) \cap E) \le C r^{p-1}$ with $C = C(E)$ and $0 < r \le r_0$.
\end{lemma}

\noindent Only the case $p=2$ is used below (Corollary \ref{cor:7.6}); the lemma is stated and proved for every $p\ge1$.

\noindent We use Theorem \ref{thm:7.3} on the decomposition into $L$-analytic pieces and Theorem \ref{thm:7.4}, which guarantees a Whitney condition for subanalytic sets, to derive this corollary on the $(p-1)$-dimensional volume of $S(0, r) \cap E$ with $E$ subanalytic in $\mathbb{R}^n$ of dimension $p$: we give a direct proof of this fact. The same estimate also follows from the Lipschitz conic structure theorem for subanalytic germs, proved later by G. Valette \cite[Thm.~3.5]{Va} by different methods.

This corollary, of interest only when $0$ is a singular point of $E$ (in other cases it is trivial), is false if we replace the hypothesis ``subanalytic'' with the weaker hypothesis ``Whitney stratified set'' (i.e.\ adjacent pairs of strata satisfy the classical Whitney condition of \S6.8), as shown in the counterexample described in \S6.7.

A corollary is that every subanalytic set is locally $v$-regular, i.e.\ $\text{Volume } (B(0, r) \cap E) \le C r^p$, but this latter result is obtained more easily as in \cite[Thm.~IV.2]{FE}. In the same article other results are also proved that provide a more complete and interesting picture of the situation.

\begin{proof}
\noindent\textit{Notation.} $\mathcal H^k$ is the $k$-dimensional Hausdorff measure, $\mathrm{Volume}$ in the statement means $\mathcal H^{p-1}$ (for the rectifiable sets met below this is the $(p-1)$-dimensional volume), $\#$ is the cardinality, and $\omega_k$ is the volume of the unit ball of $\mathbf R^k$. If $p=0$, $E$ is finite near $0$ and $S(0,r)\cap E=\emptyset$ for small $r$; so we assume $p\ge1$.

\noindent\textit{Outline of the proof.}

\smallskip\noindent\textit{Strategy.}
\emph{(i)} By Stasica's theorem (Theorem \ref{thm:7.3}) $E$ is, near $0$, a finite union of graphs $G=\{(x,\varphi(x)):x\in \Omega\}$ of analytic maps with bounded first derivatives, defined on open subanalytic sets $\Omega$ of a coordinate space; the volume is estimated on each graph separately.
\emph{(ii)} On a graph the function $\rho(x)=|(x,\varphi(x))|$ measures the distance from $0$ of the point of $G$ above $x$. The Whitney condition at $0$, which holds for subanalytic sets by Hironaka's theorem (Theorem \ref{thm:7.4}), implies via Lemma \ref{lem:4.5} that $\rho$ has no critical point near $0$. Then the level set $Z_r=\{\rho=r\}\subset \Omega$ is a hypersurface of $\Omega$ and $S(0,r)\cap G$ is its image under $x\mapsto(x,\varphi(x))$.
\emph{(iii)} The map $x\mapsto(x,\varphi(x))$ lengthens $(p-1)$-dimensional volumes by a factor at most $\Lambda_\Phi^{p-1}$, with $\Lambda_\Phi$ fixed by the derivative bound. So it suffices to prove that $\mathcal H^{p-1}(Z_r)$ is bounded by a constant times $r^{p-1}$.
\emph{(iv)} This is obtained by projecting $Z_r$ along each coordinate axis $e_i$: the area formula bounds $\int_{Z_r}|\nu_i|\,d\mathcal H^{p-1}$ by the integral over the projected $(p-1)$-plane of the number of points of $Z_r$ on each line parallel to $e_i$, $\nu$ being the unit normal of $Z_r$. Since $\sum_i|\nu_i|\ge1$, summing over $i$ bounds $\mathcal H^{p-1}(Z_r)$.
\emph{(v)} The number of points on a line is bounded by a constant $\mathfrak N$ independent of $r$ and of the line: the number of connected components of the fibres of a subanalytic map is bounded (Bierstone--Milman), and for almost every line the components are points.
\emph{(vi)} Because $Z_r\subset B(0,r)$, the projection of $Z_r$ lies in a ball of radius $r$ of $\mathbf R^{p-1}$, which has volume $\omega_{p-1}r^{p-1}$. This gives the power $r^{p-1}$.

\smallskip\noindent\textit{Geometric meaning.}
$Z_r$ is the set of parameters $x\in \Omega$ whose image on $G$ is at distance $r$ from $0$; it lies in the ball $B(0,r)$ of the parameter space because $|x|\le|(x,\varphi(x))|$. A line parallel to $e_i$ meets $Z_r$ transversally at $x$ exactly when $\nu_i(x)\ne0$, i.e.\ when $e_i$ is not tangent to $Z_r$ at $x$; $\nu_i$ is the Jacobian of the projection of $Z_r$ along $e_i$. The non-transversal points are the critical points of that projection, and their images are the critical values, a null set (area formula, a Sard-type statement in equal dimensions). Whitney's condition prevents $Z_r$ from degenerating, in the sense that $\nabla\rho\ne0$; the finiteness theorem of Bierstone--Milman prevents it from oscillating, in the sense that lines cut it in at most $\mathfrak N$ points.

\smallskip\noindent\textit{Results used without proof.}
\begin{enumerate}
\item[(T1)] \emph{Stasica} \cite{S}, Theorem \ref{thm:7.3}: a relatively compact subanalytic set is a finite union of $L$-analytic pieces (Definition \ref{def:7.2}). \emph{Used in:} Step 0.
\item[(T2a)] \emph{Hironaka} \cite{HS}: a subanalytic set has locally finitely many connected components, each subanalytic. \emph{Used in:} Step~1 (components $\Omega_c$ of $\Omega$).
\item[(T2b)] \emph{Hironaka} \cite{HS}: the projection of a relatively compact subanalytic set, and the image of a relatively compact subanalytic set by an analytic map, are subanalytic. \emph{Used in:} Step 0 ($\Omega=\pi(G)$ is subanalytic); Step~2(a) ($\Omega^{\flat}$ is subanalytic, and the graph of $\Psi_i$ is subanalytic).
\item[(T2c)] \emph{Hironaka} \cite{HS}: finite Boolean combinations of subanalytic sets are subanalytic. \emph{Used in:} Step 0 (removal of the point $\Phi^{-1}(0)$ from $\Omega$).
\item[(T3)] \emph{Hironaka} \cite{H2}, Theorem \ref{thm:7.4}, applied to a connected component of the graph and to the single point $\{0\}$. \emph{Used in:} Step 1.
\item[(T4)] \emph{Bierstone--Milman} \cite[Thm.~3.14]{BM88} (see also \cite{Ha75,DLS82}): if $Y_0$ is a relatively compact subset of a real analytic manifold $Q_1$ and $\Xi:Y_0\to Q_2$ is a subanalytic map into a real analytic manifold $Q_2$, the number of connected components of a fibre $\Xi^{-1}(y)$ is bounded locally on $Q_2$. \emph{Used in:} Step 2(a) (bound $\mathcal \mathfrak N$ on the number of components of $Z_r\cap\lambda_i(w)$) and Step 3, case $p=1$.
\item[(T5a)] \emph{Area formula in $\mathbf R^\mu$} \cite[Thm.~3.2.3]{Fed}: if $f:\mathbf R^\mu\to\mathbf R^{\mu'}$ is Lipschitz with $\mu\le\mu'$ and $A\subset\mathbf R^\mu$ is Lebesgue measurable, then $\int_AJ_\mu f\,d\mathcal L^\mu=\int_{\mathbf R^{\mu'}}\#\bigl(A\cap f^{-1}(y)\bigr)\,d\mathcal H^\mu(y)$, where $J_\mu f$ is the $\mu$-dimensional Jacobian and $\#$ denotes cardinality. \emph{Used in:} not applied on its own in Steps~2--3; it is the Euclidean case from which \cite[Cor.~3.2.20]{Fed} is deduced (see (T5b)).
\item[(T5b)] \emph{Area formula on rectifiable sets} \cite[Cor.~3.2.20]{Fed} (with the standing hypotheses of 3.2.19): if $Y$ is an $(\mathcal H^\mu,\mu)$ rectifiable and $\mathcal H^\mu$ measurable subset of a Euclidean space (rectifiable in the sense of 3.2.14 includes $\mathcal H^\mu(Y)<\infty$), $f:Y\to\mathbf R^{\mu'}$ is Lipschitz and $\mu\le\mu'$, then
\[
\int_Y(g\circ f)\,\operatorname{ap}J_\mu f\,d\mathcal H^\mu=\int_{\mathbf R^{\mu'}}g(z)\,\#\bigl(Y\cap f^{-1}(z)\bigr)\,d\mathcal H^\mu(z)\qquad\text{for every function }g,
\]
where $\operatorname{ap}J_\mu f=\|\textstyle\bigwedge_\mu\operatorname{ap}Df\|$ is the approximate Jacobian, defined $\mathcal H^\mu$-almost everywhere on $Y$ by 3.2.19.

\emph{Form used.} Take $g\equiv1$, $Y$ a Borel subset of finite $\mathcal H^\mu$ measure of a $C^1$ submanifold $Z$ of dimension $\mu$, and $f=F|_Y$ with $F$ a $C^1$ map on a neighbourhood of $Y$, Lipschitz on $Y$. At each $x\in Y$, $\mathrm{Tan}^\mu(\mathcal H^\mu\llcorner Y,x)$ is $T_xZ$ for $\mathcal H^\mu$-almost every $x$ (3.2.17, 3.2.19, with a $C^1$ chart for $\psi$), so $\operatorname{ap}J_\mu f(x)$ is the Jacobian $J_\mu(DF(x)|_{T_xZ})$. Then
\begin{equation}\tag{7.5.1}\label{eq:75:T5}
\int_YJ_\mu(DF|_{TZ})\,d\mathcal H^\mu=\int\#\bigl(Y\cap F^{-1}(z)\bigr)\,d\mathcal H^\mu(z),
\end{equation}
with $\mathcal H^\mu$ equal to Lebesgue measure when the target has dimension $\mu$. If $\mathcal H^\mu(Y)$ is not known to be finite, write $Y$ as an increasing union of compact subsets $Y_h$ (each has finite measure and $F$ is Lipschitz on it), apply \eqref{eq:75:T5} to each, and let $h\to\infty$: both sides increase, and monotone convergence gives \eqref{eq:75:T5} for $Y$ (with value $+\infty$ allowed). For $F$ injective on $Y$, \eqref{eq:75:T5} reads $\mathcal H^\mu(F(Y))=\int_YJ_\mu(DF|_{TZ})\,d\mathcal H^\mu$. \emph{Used in:} Step~2(b) (tangency set), Step~3 (volume of $Z_r$ and passage from $Z_r$ to $S(0,r)\cap G$).
\end{enumerate}
Lemma \ref{lem:4.5} is proved in \S\ref{sec:whitney} above and is used as stated.

\medskip\noindent\textit{Step 0: notation and reduction to a single graph.}

Put $E_0:=E\cap B(0,2)$, a relatively compact subanalytic set; $S(0,r)\cap E=S(0,r)\cap E_0$ for $r<2$. By (T1), $E_0=G_1\cup\dots\cup G_\beta$, each $G_b$ an $L$-analytic piece. Volumes are invariant under orthogonal maps of $\mathbf R^n$ and such maps preserve $S(0,r)$, subanalyticity and the Whitney condition; so, rotating each piece separately, we may assume that $G_b$ has the form
\begin{gather*}
G=\{(x,\varphi(x)):x\in \Omega\},\quad \Omega\subset\mathbf R^q\ \text{open, bounded},\\ \varphi=(\varphi^1,\dots,\varphi^{n-q}):\Omega\to\mathbf R^{n-q}\ \text{analytic},
\end{gather*}
with $\bigl|\partial\varphi^j/\partial x_k\bigr|\le K^j_k$ on $\Omega$ for constants $K^j_k\ge0$, where $0\le q\le p$ because $\dim E=p$. Set
\[
\Phi(x):=(x,\varphi(x)),\qquad \kappa:=\Bigl(\sum_{j,k}(K^j_k)^2\Bigr)^{1/2},\qquad \Lambda_\Phi:=\sqrt{1+\kappa^2},\qquad \rho(x):=|\Phi(x)| .
\]
Then $G=\Phi(\Omega)$, $\Phi$ is an analytic injective immersion, and $\pi\circ\Phi=\mathrm{id}_\Omega$ for the projection $\pi(x,y)=x$, so $\Phi:\Omega\to G$ is a homeomorphism. For $v\in\mathbf R^q$ one has $|D\Phi(x)v|^2=|v|^2+|D\varphi(x)v|^2$ and $\|D\varphi(x)\|^2\le\sum_{j,k}(\partial_k\varphi^j(x))^2\le\kappa^2$, hence
\begin{equation}\tag{7.5.2}\label{eq:75:lip}
|v|\le|D\Phi(x)v|\le\Lambda_\Phi|v|\qquad(x\in \Omega,\ v\in\mathbf R^q).
\end{equation}
Also $|x|\le\rho(x)$ for $x\in \Omega$. The set $\Omega=\pi(G)$ is subanalytic by (T2b). Since $\Phi$ is injective, $\Phi^{-1}(0)$ has at most one point; removing it from $\Omega$ leaves, by (T2c), an open bounded subanalytic set and removes from $G$ at most the point $0$, which lies on no sphere $S(0,r)$, $r>0$. We keep the names $\Omega$, $G$, and now $\rho>0$ on $\Omega$, so $\rho$ is analytic on $\Omega$.

It is enough to prove, for each piece $G$, the existence of $r_G>0$ and $C_G$ with
\begin{equation}\tag{7.5.3}\label{eq:75:piece}
\mathcal H^{p-1}\bigl(S(0,r)\cap G\bigr)\le C_G\,r^{p-1}\qquad(0<r\le r_G);
\end{equation}
the lemma then holds with $r_0=\min(1,\min_br_{G_b})$ and $C=\sum_bC_{G_b}$.

\medskip\noindent\textit{Step 1: $\rho$ has no critical point near $0$.}

By (T2a), $\Omega$ has finitely many connected components $\Omega_1,\dots,\Omega_m$, and $M_c:=\Phi(\Omega_c)$ is a connected real analytic submanifold of $\mathbf R^n$, subanalytic in $\mathbf R^n$ (it equals $G\cap(\Omega_c\times\mathbf R^{n-q})$), with $0\notin M_c$.

\emph{Case $0\notin\overline{M_c}$.} Then $\delta_c:=\operatorname{dist}(0,M_c)>0$.

\emph{Case $0\in\overline{M_c}$.} Apply Theorem \ref{thm:7.4} to $M_c$ and to the set $\{0\}$, which is subanalytic, closed and contained in $\overline{M_c}-M_c$ (item (T3)). Let $\Sigma$ be the exceptional closed set given by the theorem; $\Sigma\subseteq\{0\}$, so $\Sigma=\emptyset$ or $\Sigma=\{0\}$, and in the latter case $\{0\}-\Sigma=\emptyset$ would not be dense in $\{0\}$. Hence $\Sigma=\emptyset$ and $(M_c,\{0\})$ satisfies a Whitney condition at $0$. By Lemma \ref{lem:4.5} there is $\varepsilon_c>0$ such that $d|_{M_c}$ has no critical point on $M_c\cap B(0,\varepsilon_c)$.

Let $r_1>0$ be smaller than every $\delta_c$ and $\varepsilon_c$ occurring. If $x\in \Omega_c$ and $\rho(x)<r_1$, then $0\in\overline{M_c}$ and $\Phi(x)\in B(0,\varepsilon_c)$. Since $\Phi:\Omega_c\to M_c$ is a diffeomorphism and $\rho=d\circ\Phi$, $x$ is not a critical point of $\rho$. Thus
\begin{equation}\tag{7.5.4}\label{eq:75:nocrit}
\nabla\rho(x)\ne0\qquad\text{for all }x\in \Omega\text{ with }\rho(x)<r_1 .
\end{equation}
For $0<r<r_1$ put
\[
Z_r:=\{x\in \Omega:\rho(x)=r\},\qquad \nu:=\frac{\nabla\rho}{|\nabla\rho|}\ \text{ on }Z_r .
\]
By \eqref{eq:75:nocrit}, $Z_r$ is a real analytic hypersurface of $\Omega$ (empty or of dimension $q-1$), $\nu$ is its unit normal, and, because $G=\Phi(\Omega)$,
\begin{equation}\tag{7.5.5}\label{eq:75:slice}
S(0,r)\cap G=\Phi(Z_r),\qquad Z_r\subset\{x\in\mathbf R^q:|x|\le r\},
\end{equation}
the inclusion following from $|x|\le\rho(x)=r$.

\emph{Pieces with $q<p$.} If $q=0$, $G$ is a point (or empty after the removal of $0$), and it misses $S(0,r)$ for $r<|G|$. If $1\le q<p$, $\Phi(Z_r)$ is a $C^1$ manifold of dimension $q-1\le p-2$, so $\mathcal H^{p-1}(\Phi(Z_r))=0$. In both cases \eqref{eq:75:piece} holds with $C_G=0$, for $r_G>0$ small enough (for $q=0$, $r_G<|G|$ if $G\ne\{0\}$ was kept).

From now on $q=p$.

\medskip\noindent\textit{Step 2: a uniform bound for the intersections with lines.}

\smallskip\noindent\textit{Notation.}
Let $e_1,\dots,e_p$ be the standard basis of $\mathbf R^p$, $\nu_i:=\langle\nu,e_i\rangle$, and $P_i:\mathbf R^p\to\mathbf R^{p-1}$ the map forgetting the $i$-th coordinate. For $w\in\mathbf R^{p-1}$ and $s\in\mathbf R$ let $[w,s]_i\in\mathbf R^p$ be the point with $i$-th coordinate $s$ and with the remaining coordinates $w$. The line $\lambda_i(w):=P_i^{-1}(w)=\{[w,s]_i:s\in\mathbf R\}$ is parallel to $e_i$. For $0<r<r_1$ put
\[
Z_r^{(i)}:=\{x\in Z_r:\nu_i(x)=0\}.
\]
Geometrically, $\nu_i(x)=0$ means $e_i\in T_xZ_r=\nu(x)^\perp$: $\lambda_i(P_ix)$ is tangent to $Z_r$ at $x$.

\smallskip\noindent\textit{(a) Components of the fibres are bounded in number.}
Let $\Omega^{\flat}:=\{x\in \Omega:\rho(x)<r_1\}=\pi(G\cap B(0,r_1))$; it is relatively compact in $\mathbf R^p$ and subanalytic by (T2b). Define
\[
\Psi_i:\Omega^{\flat}\to\mathbf R^{p-1}\times\mathbf R,\qquad \Psi_i(x):=\bigl(P_ix,\ \rho(x)^2\bigr).
\]
$\Psi_i$ is continuous, and its graph is the image of the relatively compact subanalytic set $G\cap B(0,r_1)$ under an analytic map, hence subanalytic by (T2b). Its fibre over $(w,r^2)$ is
\[
\Psi_i^{-1}(w,r^2)=Z_r\cap\lambda_i(w).
\]
By (T4) the number of connected components of the fibres is bounded locally on $\mathbf R^p=\mathbf R^{p-1}\times\mathbf R$; it is therefore bounded on the compact set $\{|w|\le r_1,\ 0\le t\le r_1^2\}$ of parameters $(w,t)$. Doing this for $i=1,\dots,p$ and taking the largest bound we get an integer $\mathfrak N$ with
\begin{equation}\tag{7.5.6}\label{eq:75:Nbound}
\#\{\text{connected components of }Z_r\cap\lambda_i(w)\}\le \mathfrak N\qquad(0<r<r_1,\ w\in\mathbf R^{p-1},\ 1\le i\le p).
\end{equation}
(For $|w|>r$ the fibre is empty, by \eqref{eq:75:slice}.)

\smallskip\noindent\textit{(b) The tangency points project onto a null set.}
Fix $i$ and $r\in(0,r_1)$. The Jacobian of $P_i|_{Z_r}$ at $x$ is $|\nu_i(x)|$. Indeed, if $\nu=\pm e_i$ then $P_i$ maps $\nu^\perp=e_i^\perp$ isometrically and $|\nu_i|=1$. Otherwise let $t:=(e_i-\nu_i\nu)/\sqrt{1-\nu_i^2}$, a unit vector of $\nu^\perp$; the space $\nu^\perp$ is the orthogonal sum of $\nu^\perp\cap e_i^\perp$, which $P_i$ fixes pointwise, and of $\mathbf R t$; and $P_it=t-\langle t,e_i\rangle e_i$ is orthogonal to $\nu^\perp\cap e_i^\perp$ with $|P_it|^2=1-\langle t,e_i\rangle^2=\nu_i^2$.

Apply the area formula (T5b), in the form \eqref{eq:75:T5}, to $P_i|_{Z_r}$, whose source and target have the same dimension $p-1$:
\begin{equation}\tag{7.5.7}\label{eq:75:area1}
\int_{\mathbf R^{p-1}}\#\bigl(Z_r^{(i)}\cap\lambda_i(w)\bigr)\,dw=\int_{Z_r^{(i)}}|\nu_i|\,d\mathcal H^{p-1}=0 .
\end{equation}
The integrand on the left is a non-negative measurable function with integral $0$, so it vanishes for almost every $w$: for a.e.\ $w$ the line $\lambda_i(w)$ has no point of tangency with $Z_r$. (This is the case of equal dimensions of Sard's theorem.)

\smallskip\noindent\textit{(c) For a.e.\ $w$ the components of the fibre are points.}
Let $w$ be such that $\lambda_i(w)\cap Z_r^{(i)}=\emptyset$, and let $x=[w,s_0]_i\in Z_r\cap\lambda_i(w)$. Put $g(s):=\rho([w,s]_i)-r$, defined for $s$ near $s_0$ ($\Omega$ is open). Then $x$ corresponds to the zero $s_0$ of $g$, and
\[
g'(s_0)=\langle\nabla\rho(x),e_i\rangle=|\nabla\rho(x)|\,\nu_i(x)\ne0 ,
\]
by \eqref{eq:75:nocrit} and $\nu_i(x)\ne0$. By continuity of $g'$, $g$ is strictly monotone near $s_0$, so $s_0$ is the only zero of $g$ there: $x$ is an isolated point of $Z_r\cap\lambda_i(w)$. A connected subset of $Z_r\cap\lambda_i(w)$ with two distinct points would split into two disjoint non-empty open subsets; hence the connected components of $Z_r\cap\lambda_i(w)$ are points, and by \eqref{eq:75:Nbound}
\begin{equation}\tag{7.5.8}\label{eq:75:count}
\#\bigl(Z_r\cap\lambda_i(w)\bigr)\le \mathfrak N\qquad\text{for a.e. }w\in\mathbf R^{p-1},\ 0<r<r_1,\ 1\le i\le p .
\end{equation}

\medskip\noindent\textit{Step 3: the volume of the slice.}

\smallskip\noindent\textit{The case $p=1$.}
Here $Z_r\subset \Omega\subset\mathbf R$ is discrete by \eqref{eq:75:nocrit}. Applying (T4) to the subanalytic map $\rho^2:\Omega^{\flat}\to\mathbf R$, whose fibre over $r^2$ is $Z_r$, gives $\#Z_r\le \mathfrak N$ for $0<r<r_1$ ($\mathfrak N$ taken as in (a) for this map). As $\Phi$ is injective, $\mathcal H^0(S(0,r)\cap G)=\#Z_r\le \mathfrak N$, which is \eqref{eq:75:piece} with $C_G=\mathfrak N$, since $r^{p-1}=1$.

\smallskip\noindent\textit{The case $p\ge2$: volume of $Z_r$.}
Since $|\nu_i|\le1$, $\sum_{i=1}^p|\nu_i|\ge\sum_{i=1}^p\nu_i^2=1$. By the area formula (T5b), in the form \eqref{eq:75:T5}, applied to $P_i|_{Z_r}$ (a map between manifolds of dimension $p-1$), by \eqref{eq:75:count} and by $P_i(Z_r)\subset\{|w|\le r\}\subset\mathbf R^{p-1}$, which follows from \eqref{eq:75:slice},
\begin{equation}\tag{7.5.9}\label{eq:75:vol}
\mathcal H^{p-1}(Z_r)\le\sum_{i=1}^p\int_{Z_r}|\nu_i|\,d\mathcal H^{p-1}=\sum_{i=1}^p\int_{\mathbf R^{p-1}}\#\bigl(Z_r\cap\lambda_i(w)\bigr)\,dw\le p\,\mathfrak N\,\omega_{p-1}\,r^{p-1}.
\end{equation}

\smallskip\noindent\textit{From $Z_r$ to the slice of $G$.}
For $x\in Z_r$ let $\ell_x:=D\Phi(x)|_{T_xZ_r}$, a linear map from the $(p-1)$-dimensional space $T_xZ_r$ to $\mathbf R^n$; explicitly $\ell_xv=(v,D\varphi(x)v)$, the tangent vector at $\Phi(x)$ to the image of a curve with velocity $v$. Its $(p-1)$-dimensional Jacobian is $J(x)=\sigma_1\cdots\sigma_{p-1}$, where $\sigma_1,\dots,\sigma_{p-1}$ are the singular values of $\ell_x$, i.e.\ the square roots of the eigenvalues of $\ell_x^{*}\ell_x$. By \eqref{eq:75:lip}, $\sigma_a\le\Lambda_\Phi$ for every $a$, hence $J(x)\le\Lambda_\Phi^{p-1}$. The map $\Phi|_{Z_r}$ is injective, so by (T5b), in the form \eqref{eq:75:T5}, and \eqref{eq:75:slice}
\[
\mathcal H^{p-1}\bigl(S(0,r)\cap G\bigr)=\int_{Z_r}J\,d\mathcal H^{p-1}\le\Lambda_\Phi^{p-1}\,\mathcal H^{p-1}(Z_r).
\]
Together with \eqref{eq:75:vol},
\[
\mathcal H^{p-1}\bigl(S(0,r)\cap G\bigr)\le p\,\mathfrak N\,\omega_{p-1}\,\Lambda_\Phi^{p-1}\,r^{p-1}\qquad(0<r<r_1),
\]
which is \eqref{eq:75:piece} with $C_G=p\,\mathfrak N\,\omega_{p-1}\Lambda_\Phi^{p-1}$ and $r_G=r_1$.

\medskip\noindent\textit{Conclusion.}

Each of the finitely many pieces satisfies \eqref{eq:75:piece}: those with $q<p$ for $C_G=0$, those with $q=p$ with the constants of Step~3. By the reduction of Step 0 the lemma holds with $r_0=\min(1,\min_br_{G_b})$ and $C=\sum_bC_{G_b}$. 
\end{proof}

\begin{corollary}\label{cor:7.6}
Every isolated singularity of an orientable subanalytic set $E$ with $\operatorname{dim} E = 2$ is conformally point-like. In particular this holds for $E$ semianalytic and $E$ analytic.
\end{corollary}

\begin{proof}
Let $0$ be an isolated singular point of the orientable subanalytic set $E$, $\dim E=2$, and let $M$ be a connected component of $E-\{0\}$ near $0$; by hypothesis $M$ is an orientable $\mathcal{C}^2$ (indeed real-analytic) $2$-manifold with $0\in\overline M-M$. By Theorem \ref{thm:5.3} it suffices to show that $\{0\}$ is an $M$-regular singularity of $E$, i.e.\ that the two conditions of Definition \ref{def:5.2} hold.

\noindent\textit{Condition 2) (metric decay).} Apply Lemma \ref{lem:7.5} with $p=\dim E=2$: $\mathrm{Volume}(S(0,r)\cap E)\le Cr^{p-1}=Cr$ for $r\le r_0$. Since $\dim E=2$, $S(0,r)\cap E$ has dimension $1$, so its $(p-1)$-dimensional volume is exactly its length $L(S(0,r)\cap E)$. This is precisely condition 2) of Definition \ref{def:5.2}.

\noindent\textit{Condition 1) (Whitney condition).} Apply Theorem \ref{thm:7.4} with $N=\{0\}$: a single point is trivially subanalytic in $\mathbb{R}^n$ and closed, hence locally closed, so Theorem \ref{thm:7.4} furnishes a closed subanalytic set $S\subseteq N=\{0\}$ such that $N-S$ is dense in $N$ and $(M,N)$ satisfies a Whitney condition on $N-S$. Since $N=\{0\}$ is a single point, $S$ is either $\emptyset$ or $\{0\}$; but $S=\{0\}$ gives $N-S=\emptyset$, which is \emph{not} dense in $N=\{0\}$ (its closure is $\emptyset\ne\{0\}$), contradicting property 2) of Theorem \ref{thm:7.4}. Hence necessarily $S=\emptyset$, so $N-S=\{0\}$ and property 3) of Theorem \ref{thm:7.4} gives that $(M,\{0\})$ satisfies a Whitney condition \emph{at $0$ itself} -- not merely at a dense subset of singular points, which is all that Theorem \ref{thm:7.4} promises in general, but the strongest possible statement for the degenerate case of a single-point stratum. This is precisely condition 1) of Definition \ref{def:5.2}.

Both conditions of $M$-regularity hold, so by Theorem \ref{thm:5.3}, $\{0\}$ is a conformally point-like singularity of $E$. Since a semianalytic set is subanalytic (take the projection in Definition \ref{def:7.1} to be the identity map) and an analytic set is semianalytic, the last sentence of the statement follows immediately.
\end{proof}

\section*{Acknowledgements}
The original preprint appeared in 1989 at the Dipartimento di Matematica, Universit\`a di Pisa. The author thanks Prof.~Riccardo Benedetti and Prof.~Fulvio Lazzeri for the ideas and discussions supporting the original research thesis work.

\vfill

\newpage

\printbibliography

@book{	AS,
  author    = 	 {Ahlfors, L. V. and Sario, L.},
  title     = 	 {Riemann Surfaces},
  publisher = 	 {Princeton University},
  year      = 	 {1960},
shorthand =	{AS},
}

@book{	C,
  author    = 	 {Caraman, Petru},
  title     = 	 {N-Dimensional Quasiconformal Mappings},
  publisher = 	 {Abacus Press},
  address   = 	 {England},
  year      = 	 {1974},
shorthand =	{C},
}

@article{	Ch,
author = 	 {Chern, Shiing-Shen},
title =	{An elementary proof of the existence of isothermal parameters on a surface},
journal =	 {Proc. Amer. Math. Soc.},
volume =	 {6},
number =	 {},
year =	 {1955},
pages =	 {771},
shorthand =	{Ch},
}

@book{	FK,
  author    = 	 {Farkas, Hershel M. and Kra, Irwin},
  title     = 	 {Introduction to Riemann Surfaces},
  publisher = 	 {Springer}, 
  address   = 	 {Berlin}, 
  year      = 	{1980},
shorthand =	{FK},
}

@article{	FE,
author = 	 {Ferrarotti, Massimo},
title =	 {Volume on Stratified Sets},
journal =	 {Ann. Mat. Pura ed Appl. (IV)},
volume =	 {CXLIV},
pages =	 {183},
year =	 {1986},
shorthand =	{FE},
}

@article{	H1,
author = 	 {Hironaka, Heisuke},
title =	 {Introduction to Real Analytic Sets and Real Analytic Maps},
journal =	 {Quad. CNR-Istituto Matematico dell'Universit\`a di Pisa},
year =	{1973},
shorthand =	{H1},
}

@incollection{	H2,
author = 	 {Hironaka, Heisuke},
title =	 {Stratification and Flatness},
booktitle =	 {Real and Complex Singularities},
booksubtitle =	 {Proceedings of the Ninth Nordic Summer School/NAVF Symposium in Mathematics, Oslo, August 5--25, 1976},
editor =	 {Holm, P.},
publisher =	 {Sijthoff \& Noordhoff},
year =	 {1977},
pages =	 {199--265},
shorthand =	{H2},
}

@book{	Hi,
  author    = 	 {Hirsch, Morris W.},
  title     = 	 {Differentiable Topology},
series =	 {Graduate Texts in Mathematics},
  publisher = 	 {Springer},
  address   = 	 {Berlin},
  year      = 	{1976},
shorthand =	{Hi},
}

@article{	S,
author = 	 {Stasica, Jacek},
title =	 {Whitney Property for Subanalytic Sets},
journal =	 {Zeszyty Nauk. Univ. Jag.},
number =	 {DCXXIII},
year =	 {1982},
pages =	 {211},
shorthand =	{S},
}

@article{	Su,
author = 	 {Suominen, Kalevi},
title =	 {Quasiconformal Maps in Manifolds},
journal =	 {Ann. Acad. Sci. Fennicae},
number =	 {AI 393},
year =	 {1966},
pages =	 {39},
shorthand =	{Su},
}

@book{	V,
  author    = 	 {V\"ais\"al\"a, Jussi},
  title     = 	 {Lectures on N-Dimensional Quasiconformal Mappings},
  publisher = 	 {Springer-Verlag},
  year      = 	 {1971},
shorthand =	{V},
}

@article{	W,
author = 	 {Whitney, Hassler},
title =	 {Tangents to an Analytic Variety},
journal =	 {Ann. of Math.},
volume =	 {81},
pages =	 {496},
Year=	 {1965},
shorthand = {W},
}

@incollection{	HS,
author = 	 {Hironaka, Heisuke},
title =	 {Subanalytic Sets},
booktitle =	 {Number Theory, Algebraic Geometry and Commutative Algebra, in honor of Yasuo Akizuki},
publisher =	 {Kinokuniya},
address =	 {Tokyo},
year =	 {1973},
pages =	 {453--493},
shorthand =	 {HS},
}

@article{	Ma,
author = 	 {Mather, John N.},
title =	 {Notes on Topological Stability},
journal =	 {Bull. Amer. Math. Soc. (N.S.)},
volume =	 {49},
number =	 {4},
year =	 {2012},
pages =	 {475--506},
addendum =	 {Reprint of the original 1970 booklet, Harvard University Press},
shorthand =	 {Ma},
}

@article{	Th,
author = 	 {Thom, Ren\'e},
title =	 {Ensembles et morphismes stratifi\'es},
journal =	 {Bulletin of the American Mathematical Society},
volume =	 {75},
pages =	 {240--284},
year =	 {1969},
shorthand =	 {Th},
}

@article{	NV,
author = 	 {Nocera, Guglielmo and Volpe, Marco},
title =	 {Whitney stratifications are conically smooth},
journal =	 {Selecta Mathematica (New Series)},
volume =	 {29},
pages =	 {68},
year =	 {2023},
shorthand =	 {NV},
}

@unpublished{	Mo89,
author = 	 {Mosca, Flavio},
title =	 {Un problema per le singolarit\`a isolate di superficie},
note =	 {Preprint 1.7 (441), Dipartimento di Matematica, Sezione di Geometria e Algebra, Universit\`a di Pisa. The original preprint of which this is a 2026 revised and annotated edition.},
date =	 {1989-04},
shorthand =	 {Mo89},
}

@book{	dC,
author = 	 {do Carmo, Manfredo P.},
title =	 {Differential Geometry of Curves and Surfaces},
publisher =	 {Prentice-Hall},
address =	 {Englewood Cliffs, NJ},
year =	 {1976},
shorthand =	 {dC},
}

@book{	Fed,
author = 	 {Federer, Herbert},
title =	 {Geometric Measure Theory},
series =	 {Grundlehren der mathematischen Wissenschaften},
volume =	 {153},
publisher =	 {Springer-Verlag},
year =	 {1969},
shorthand =	 {Fed},
}

@article{	BM88,
author = 	 {Bierstone, Edward and Milman, Pierre D.},
title =	 {Semianalytic and subanalytic sets},
journal =	 {Inst. Hautes \'Etudes Sci. Publ. Math.},
volume =	 {67},
pages =	 {5--42},
year =	 {1988},
shorthand =	 {BM88},
}

@article{	Ha75,
author = 	 {Hardt, Robert M.},
title =	 {Stratification of real analytic mappings and images},
journal =	 {Invent. Math.},
volume =	 {28},
pages =	 {193--208},
year =	 {1975},
shorthand =	 {Ha75},
}

@article{	DLS82,
author = 	 {Denkowska, Zofia and {\L}ojasiewicz, Stanis{\l}aw and Stasica, Jacek},
title =	 {Sur le nombre des composantes connexes de la section d'un sous-analytique},
journal =	 {Bull. Acad. Polon. Sci. S\'er. Sci. Math.},
volume =	 {30},
pages =	 {333--335},
year =	 {1982},
shorthand =	 {DLS82},
}

@article{	Va,
author = 	 {Valette, Guillaume},
title =	 {Poincar\'e duality for $L^p$ cohomology on subanalytic singular spaces},
journal =	 {Math. Ann.},
volume =	 {380},
pages =	 {789--823},
year =	 {2021},
doi =	 {10.1007/s00208-021-02151-4},
shorthand =	 {Va},
}

\vfill

\end{document}